\documentclass[11pt]{amsart}
\usepackage[british]{babel}

\usepackage{a4wide}
\usepackage{amssymb}
\usepackage{amsmath}
\usepackage{amsthm}
\usepackage{array}
\usepackage{enumitem}
\usepackage[hidelinks]{hyperref}
\usepackage{breakurl}
\usepackage[style=numeric,backref=true,
isbn=false,doi=true,backend=bibtex,style=alphabetic]{biblatex}
\newtheorem{thm}{Theorem}[section]

\newtheorem{cor}[thm]{Corollary}
\newtheorem{lem}[thm]{Lemma}

\newtheorem{prop}[thm]{Proposition}

\theoremstyle{definition}

\newtheorem{definition}[thm]{Definition}

\newtheorem{remark}[thm]{Remark}

\renewcommand{\epsilon}{\varepsilon}
\renewcommand{\phi}{\varphi}
\newcommand{\defeq}{\mathrel{\mathop:}=}

\DeclareMathOperator{\dom}{dom}
\DeclareMathOperator{\St}{St}

\DeclareMathOperator{\Aut}{Aut}
\DeclareMathOperator{\Sym}{Sym}

\DeclareMathOperator{\id}{id}

\begin{document}
%%%%%%%%%%%%%%%%%%%%%%%%%%%%%%%%

\setlist{noitemsep}

\author{Friedrich Martin Schneider}
\address{F. M. Schneider, Institute of Algebra, TU Dresden, 01062 Dresden, Germany }
\title{Asymptotic uniform complexity and amenability}
\date{\today}

\begin{abstract} 
  We introduce a novel quantity for general dynamical systems, which we call the \emph{asymptotic uniform complexity}. We prove an inequality relating the asymptotic uniform complexity of a dynamical system to its mean topological matching number. Furthermore, we show that the established inequality yields an exact equation for perfect Hausdorff dynamical systems. Utilizing these results, we conclude that a dynamical system is amenable if its asymptotic uniform complexity equals one, and that the converse is true for perfect Hausdorff dynamical systems, such as non-discrete Hausdorff topological groups acting on themselves. Moreover, we establish a characterization of amenability for topologically free dynamical systems on perfect Hausdorff uniform spaces by means of a similar invariant. Furthermore, we provide a sufficient criterion for amenability concerning groups of Lipschitz-automorphisms of precompact pseudo-metric spaces in terms of entropic dimension. Finally, we show that vanishing topological entropy implies the asymptotic uniform complexity to equal one.
\end{abstract}

\maketitle

%%%%%%%%%%%%%%%%%%%%%%%%%%%%%%%%%%%%%%%%%%
%%%%%%%%%%%%%%%%%%%%%%%%%%%%%%%%%%%%%%%%%%

\tableofcontents

\newpage

\section{Introduction}

Amenability ranges among the most powerful and well-recognized concepts in topological dynamics and topological group theory. A remarkable amount of literature has been dedicated to the analysis of amenable topological groups. For a more elaborate overview as well as some historical background information one may consult \cite{AnalysisOnSemigroups,brownc,runde,paterson}.

In \cite{SchneiderThom}, an invariant for general dynamical systems on uniform spaces is introduced. This quantity is defined in terms of matching numbers of a family of bipartite graphs associated with a given dynamical system and therefore called the \emph{mean topological matching number}. It is shown in \cite{SchneiderThom} that a dynamical system is amenable if its mean topological matching number equals one. Furthermore, it is proven that the converse is true provided that the considered dynamical system is Hausdorff and perfect, i.e., it does not contain any isolated points. Utilizing this result, it is shown that a Hausdorff topological group $G$ is amenable if and only if the mean topological matching number of the associated action of $G$ on itself (Definition~\ref{definition:amenable.topological.group}) equals one.

In the present paper we introduce another invariant for dynamical systems, which we call the \emph{asymptotic uniform complexity} (Definition~\ref{definition:asymptotic.uniform.complexity}). The asymptotic uniform complexity of a dynamical system is defined in terms of and therefore closely related to uniform complexities (Definition~\ref{definition:uniform.complexity}) of finite uniform coverings of the underlying uniform space. As it turns out, uniform complexities reveal some information about matchings in certain bipartite graphs associated with the considered finite uniform coverings (Lemma~\ref{lemma:uniform.refinement.map}). This suggests a connection between the asymptotic uniform complexity of a dynamical system and its mean topological matching number. Our aim is to investigate this connection and some of its consequences. In the course of the present work we provide several results in that direction (Theorem~\ref{theorem:first.main.theorem} and Theorem~\ref{theorem:second.main.theorem}). Applying the results of \cite{SchneiderThom}, we obtain some new characterizations of amenability (Corollary~\ref{corollary:amenable} and Theorem~\ref{theorem:third.main.theorem}). Particular attention is paid to the situation for topological groups (Corollary~\ref{corollary:groups} and Corollary~\ref{corollary:third.main.theorem}). Furthermore, we deduce two amenability criteria by relating the asymptotic uniform complexity to the entropic dimension (Theorem~\ref{theorem:metric.case}) and the topological entropy (Theorem~\ref{theorem:vanishing.entropy.implies.amenability}).

The paper is organized as follows. In Section~\ref{section:uniform.spaces} we recall all the necessary background knowledge regarding uniform spaces. In Section~\ref{section:amenability} we shall deal with means on function spaces in general and invariant means on dynamical systems in particular. In Section~\ref{section:matchings} we shall provide some auxiliary results about matchings in bipartite graphs, paying particular attention to bipartite graphs arising from coverings of sets. In Section~\ref{section:mean.topological.matching.number} we will summarize the main concepts and results of \cite{SchneiderThom}. In Section~\ref{section:asymptotic.uniform.complexity} we introduce the asymptotic uniform complexity, explore its connection to the mean topological matching number, and infer a corresponding characterization of amenability for perfect Hausdorff dynamical systems. Afterwards, Section~\ref{section:topologically.free.dynamical.systems} is devoted to an alternative characterization of amenability applying to topologically free dynamical systems. In Section~\ref{section:entropic.dimension} we provide a sufficient criterion for amenability concerning groups of Lipschitz-automorphisms of precompact pseudo-metric spaces in terms of entropic dimension. Finally, in Section~\ref{section:topological.entropy} we prove that vanishing topological entropy implies the asymptotic uniform complexity to equal one, which yields another amenability criterion.

\section{Uniform spaces}\label{section:uniform.spaces}

In this section we shall recall some basics concerning uniform spaces. For this purpose, we will follow the approach of \cite{Isbell}.

In order to introduce the concept of a uniform space, we shall need some set-theoretic basics. Let $X$ be a set. We denote by $\mathcal{P}(X)$ the set of all subsets of $X$. Let $\mathcal{U},\mathcal{V} \subseteq \mathcal{P}(X)$. We say that $\mathcal{V}$ \emph{refines} $\mathcal{U}$ and write $\mathcal{U} \preceq \mathcal{V}$ if \begin{displaymath}
	\forall V \in \mathcal{V} \, \exists U \in \mathcal{U} \colon \, V \subseteq U .
\end{displaymath} Furthermore, let $\mathcal{U} \wedge \mathcal{V} \defeq \mathcal{U} \cup \mathcal{V}$ and $\mathcal{U} \vee \mathcal{V} \defeq \{ U \cap V \mid U \in \mathcal{U}, \, V \in \mathcal{V} \}$. More generally, if $(\mathcal{U}_{i})_{i \in I}$ is a family of subsets of $\mathcal{P}(X)$, then we define $\bigwedge\nolimits_{i \in I} \mathcal{U}_{i} \defeq \bigcup\nolimits_{i \in I} \mathcal{U}_{i}$ and \begin{displaymath}
	\left. \bigvee\nolimits_{i \in I} \mathcal{U}_{i} \defeq \left\{ \bigcap\nolimits_{i \in I} U_{i} \, \right| (U_{i})_{i \in I} \in \prod\nolimits_{i \in I} \mathcal{U}_{i} \right\} .
\end{displaymath} For a subset $S \subseteq X$, we call $\St (S,\mathcal{U}) \defeq \bigcup \{ U \in \mathcal{U} \mid U \cap S \ne\emptyset \}$ the \emph{star} of $S$ with respect to $\mathcal{U}$. Likewise, given any $x \in X$, we call $\St (x,\mathcal{U}) \defeq \St (\{ x \},\mathcal{U})$ the \emph{star} of $x$ with respect to $\mathcal{U}$. Moreover, the \emph{star} of $\mathcal{U}$ is defined to be $\mathcal{U}^{\ast} \defeq \{ \St (U,\mathcal{U}) \mid U \in \mathcal{U} \}$. Besides, let $\mathcal{U}^{\ast, 0} = \mathcal{U}$ and $\mathcal{U}^{\ast, n+1} \defeq (\mathcal{U}^{\ast, n})^{\ast}$ for every $n \in \mathbb{N}$. We say that $\mathcal{V}$ is a \emph{star-refinement} of $\mathcal{U}$ and write $\mathcal{U} \preceq^{\ast} \mathcal{V}$ if $\mathcal{U} \preceq \mathcal{V}^{\ast}$. We shall call $\mathcal{U}$ a \emph{covering} of $X$ if $X = \bigcup \mathcal{U}$. In case $X$ is a topological space, an \emph{open covering} of $X$ is a covering of $X$ consisting entirely of open subsets of $X$. We denote by $\mathcal{C}(X)$ the set of all coverings of $X$. A \emph{uniformity} on $X$ is a non-empty subset $\mathcal{D} \subseteq \mathcal{C}(X)$ such that \begin{enumerate}
	\item[(1)] 	$\forall \mathcal{U} \in \mathcal{D} \, \forall \mathcal{V} \in \mathcal{C}(X) \colon \, \mathcal{V} \preceq \mathcal{U} \Longrightarrow \mathcal{V} \in \mathcal{D}$,
	\item[(2)]	$\forall \mathcal{U},\mathcal{V} \in \mathcal{D} \, \exists \mathcal{W} \in \mathcal{D} \colon \, \mathcal{U} \preceq^{\ast} \mathcal{W} , \, \mathcal{V} \preceq^{\ast} \mathcal{W}$.
\end{enumerate}

Now we come to uniform spaces. A \emph{uniform space} is a non-empty set $X$ equipped with a uniformity on $X$, whose elements are called the \emph{uniform coverings} of the uniform space $X$. Let $X$ be a uniform space. The set of all finite uniform coverings of $X$ shall be denoted by $\mathcal{N}(X)$. We call $X$ \emph{precompact} if, for every uniform covering $\mathcal{U}$ of $X$, there exists a finite uniform covering $\mathcal{V}$ of $X$ such that $\mathcal{V}$ refines $\mathcal{U}$. The \emph{topology of $X$} is defined as follows: a subset $S \subseteq X$ is \emph{open} in $X$ if, for every $x \in S$, there exists a uniform covering $\mathcal{U}$ of $X$ such that $\St (x,\mathcal{U}) \subseteq S$. Let $Y$ be another uniform space. A map $f \colon X \to Y$ is said to be \emph{uniformly continuous} if $f^{-1}(\mathcal{U}) \defeq \{ f^{-1}(U) \mid U \in \mathcal{U} \}$ is a uniform covering of $X$ whenever $\mathcal{U}$ is a uniform covering of $Y$. A bijection $f \colon X \to Y$ is called an \emph{isomorphism} if both $f$ and $f^{-1}$ are uniformly continuous maps. By an \emph{automorphism} of $X$, we mean an isomorphism from $X$ to itself. The group of all automorphisms of $X$ shall be denoted by $\Aut (X)$. Note that any uniformly continuous map between uniform spaces is continuous with regard to the respective topologies.

We shall frequently make use of the following observation concerning finite uniform coverings of uniform spaces.

\begin{lem}[\cite{Isbell}]\label{lemma:uniform.star.refinements} Let $X$ be a uniform space. If $\mathcal{U}$ is a finite uniform covering of $X$, then there exists a finite open uniform covering $\mathcal{V}$ of $X$ such that $\mathcal{U} \preceq^{\ast} \mathcal{V}$. \end{lem}

\begin{cor}\label{corollary:uniform.partition} Let $X$ be a uniform space. For $x \in X$, let $P(x) \defeq \bigcap \{ \St (x,\mathcal{U}) \mid \mathcal{U} \in \mathcal{N}(X) \}$. Then $\mathcal{P} \defeq \{ P(x) \mid x \in X \}$ is a partition of $X$. \end{cor}

\begin{proof} Evidently, $X = \bigcup \mathcal{P}$. We shall verify that $P(x) = P(y)$ for all $x,y \in X$ where $P(x) \cap P(y) \ne \emptyset$. To this end, let $x,y \in X$ such that $P(x) \cap P(y) \ne \emptyset$. We show that $P(x) \subseteq P(y)$. For this purpose, let $\mathcal{U} \in \mathcal{N}(X)$. By Lemma~\ref{lemma:uniform.star.refinements}, there exists $\mathcal{V} \in \mathcal{N}(X)$ such that $\mathcal{U} \preceq^{\ast} \mathcal{V}$. We argue that $\St (x,\mathcal{V}) \subseteq \St (y,\mathcal{U})$. To this end, let $V \in \mathcal{V}$ where $x \in V$. Furthermore, let $V_{0} \in \mathcal{V}$ such that $y \in V_{0}$. By assumption, there exists some $z \in P(x) \cap P(y)$. Since $z \in V \cap V_{0}$, we conclude that $V \subseteq \St(z,\mathcal{V}) \subseteq \St(V_{0},\mathcal{V}) \subseteq \St (y,\mathcal{U})$. This substantiates that $P(x) \subseteq \St (x,\mathcal{V}) \subseteq \St (y,\mathcal{U})$. Therefore, $P(x) \subseteq P(y)$. Due to symmetry, it follows that $P(x) = P(y)$. Consequently, $\mathcal{P}$ is a partition of $X$. \end{proof}

As it turns out, uniformities are closed with respect to certain kinds of excision operations.

\begin{lem}\label{lemma:excision.of.uniform.coverings.a} Let $X$ be a uniform space, let $\mathcal{U}$ and $\mathcal{V}$ be uniform coverings of $X$, and let $S \subseteq X$. Then $\mathcal{W} \defeq \{ U\setminus S \mid U \in \mathcal{U} \} \cup \{ \St (x,\mathcal{V}) \mid x \in S \}$ is a uniform covering of $X$. \end{lem}

\begin{proof} We argue that $\mathcal{U} \vee \mathcal{V}$ refines $\mathcal{W}$. To this end, let $U \in \mathcal{U}$ and $V \in \mathcal{V}$. If $V \cap S \ne \emptyset$, then there is $x \in S$ where $U \cap V \subseteq V \subseteq \St (x,\mathcal{V}) \in \mathcal{W}$. Otherwise, $U \cap V = (U \cap V)\setminus S \subseteq U\setminus S \in \mathcal{W}$. Since $\mathcal{U} \vee \mathcal{V}$ is a uniform covering of $X$, it follows that $\mathcal{W}$ is a uniform covering of $X$. \end{proof}

\begin{lem}\label{lemma:excision.of.uniform.coverings.b} Let $X$ be a uniform space, let $\mathcal{U}$ be a uniform covering of $X$, and let $F \subseteq X$ be finite. If $\mathcal{V}$ is a set of open subsets of $X$ such that $F \subseteq \bigcup \mathcal{V}$, then $\mathcal{W} \defeq \{ U\setminus F \mid U \in \mathcal{U} \} \cup \mathcal{V}$ is a uniform covering of $X$. \end{lem}

\begin{proof} Since $\mathcal{V}$ is a set of open subsets of $X$ and $F$ is a finite subset of $\bigcup \mathcal{V}$, there exists a uniform covering $\mathcal{Z}$ of $X$ such that $\{ \St (y,\mathcal{Z}) \mid y \in F \}$ refines $\mathcal{V}$. We argue that $\mathcal{U} \vee \mathcal{Z}$ refines $\mathcal{W}$. To this end, let $U \in \mathcal{U}$ and $Z \in \mathcal{Z}$. If $Z \cap F \ne \emptyset$, then there exists some $V \in \mathcal{V}$ such that $Z \subseteq V$ and hence $U \cap Z \subseteq V$. Otherwise, $U \cap Z = (U \cap Z)\setminus F \subseteq U\setminus F \in \mathcal{W}$. Since $\mathcal{U} \vee \mathcal{Z}$ is a uniform covering of $X$, this implies that $\mathcal{W}$ is a uniform covering of $X$. \end{proof}

Furthermore, we will be concerned with the following combinatorial quantity for finite uniform coverings of uniform spaces.

\begin{definition}\label{definition:uniform.complexity} Let $X$ be a uniform space. If $\mathcal{U}$ is a member of $\mathcal{N}(X)$, then the \emph{uniform complexity} of $\mathcal{U}$ is defined to be the quantity $N(\mathcal{U}) \defeq \inf \{ |\mathcal{V}| \mid \mathcal{V} \in \mathcal{N}(X), \, \mathcal{U} \preceq \mathcal{V} \}$. \end{definition}

\begin{lem}\label{lemma:point.projection.map} Let $X$ be a uniform space, let $\mathcal{U},\mathcal{V} \in \mathcal{N}(X)$ such that $N(\mathcal{U}) = |\mathcal{U}|$ and $\mathcal{U} \preceq^{\ast} \mathcal{V}$. Then there exists a (necessarily injective) map $\pi \colon \mathcal{U} \to X$ such that $\St (\pi(U),\mathcal{V}) \subseteq U$ and $\mathcal{U}\setminus \{ U \} \npreceq \{ \St (\pi(U),\mathcal{V}) \}$ for all $U \in \mathcal{U}$. \end{lem}

\begin{proof} As $X$ is non-empty, it is true that $\mathcal{U}$ refines $\mathcal{U}\setminus \{ \emptyset \}$, which implies that $\mathcal{U}\setminus \{ \emptyset \}$ is a uniform covering of $X$. Since $\mathcal{U}\setminus \{ \emptyset \}$ refines $\mathcal{U}$ and $N(\mathcal{U}) = \vert \mathcal{U} \vert$, we conclude that $\mathcal{U}$ does not contain the empty set. We are going to prove the following: \begin{displaymath}
	\forall U \in \mathcal{U} \, \exists x \in U \colon \, \mathcal{U}\setminus \{ U \} \npreceq \{ \St (x,\mathcal{V}) \} .
\end{displaymath} On the contrary, suppose that there exists some $U \in \mathcal{U}$ such that $\mathcal{U}\setminus \{ U \} \preceq \{ \St (x,\mathcal{V}) \}$ for all $x \in U$. We argue that $\mathcal{V}$ refines $\mathcal{U} \setminus \{ U \}$. To this end, let $V \in \mathcal{V}$. If $V \cap U = \emptyset$ and $V \ne \emptyset$, then clearly $\mathcal{U} \setminus \{ U \} \preceq \{ V \}$ because $\mathcal{U} \preceq \mathcal{V}$. Otherwise, $\mathcal{U}\setminus \{ U \} \preceq \{ \St(x,\mathcal{V}) \mid x \in U \} \preceq \{ V \}$. Therefore, $\mathcal{V}$ refines $\mathcal{U}\setminus \{ U \}$, which implies that $\mathcal{U}\setminus \{ U \}$ is a uniform covering of $X$. Since $\mathcal{U} \preceq \mathcal{U}\setminus \{ U \}$ and $|\mathcal{U}\setminus \{ U \}| < |\mathcal{U}|$, this contradicts our hypothesis. Consequently, there exists some $x \in U$ such that $\mathcal{U}\setminus \{ U \} \npreceq \{ \St (x,\mathcal{V}) \}$. This proves the claim. Hence, there exists a map $\pi \colon \mathcal{U} \to X$ such that $\pi (U) \in U$ and $\mathcal{U}\setminus \{ U \} \npreceq \{ \St (\pi(U),\mathcal{V}) \}$ for all $U \in \mathcal{U}$. Since $\mathcal{U} \preceq^{\ast} \mathcal{V}$, it follows that $\St (\pi (U),\mathcal{V}) \subseteq U$ for all $U \in \mathcal{U}$. Accordingly, $\pi$ is injective. \end{proof}

Finally in this section, we want to briefly discuss three important classes of uniform spaces. The first one to be mentioned is the class of pseudo-metric spaces. By a \emph{pseudo-metric space}, we mean a non-empty set $X$ equipped with a map $d \colon X \times X \to [0,\infty)$ satisfying the following conditions: \begin{enumerate}
	\item[(1)]	$\forall x \in X \colon \, d(x,x) = 0$,
	\item[(2)]	$\forall x,y \in X \colon \, d(x,y) = d(y,x)$,
	\item[(3)]	$\forall x,y,z \in X \colon \, d(x,z) \leq d(x,y) + d(y,z)$.
\end{enumerate}  If $X$ is a pseudo-metric space, then we define $B(x,r) \defeq \{ y \in X \mid d(x,y) < r \}$ whenever $x \in X$ and $r \in[0,\infty)$. Any pseudo-metric space constitutes a uniform space: if $X$ is a pseudo-metric space, then we may consider $X$ as a uniform space by equipping it with the \emph{induced uniformity}, i.e., $\{ \mathcal{U} \in \mathcal{P}(\mathcal{P}(X)) \mid \exists r \in (0,\infty) \colon \, \mathcal{U} \preceq \{ B(x,r) \mid x \in X \} \}$. This particularly applies to the space of real numbers. Concerning a uniform space $X$, we denote by $UC_{b}(X)$ the set of all bounded uniformly continuous functions from $X$ to $\mathbb{R}$.

Another example of uniform spaces is provided by the class of compact Hausdorff spaces.

\begin{prop}[\cite{ModernGeneralTopology}]\label{proposition:compact.hausdorff.spaces} If $X$ is a compact Hausdorff space, then \begin{displaymath}
	\{ \mathcal{U} \in \mathcal{P}(\mathcal{P}(X)) \mid \exists \mathcal{V} \in \mathcal{C}(X) \textit{ open}\colon \, \mathcal{U} \preceq \mathcal{V} \}
\end{displaymath} is the unique uniformity on $X$ inducing the topology of $X$. \end{prop}

As the subsequent lemma reveals, the construction from Proposition~\ref{proposition:compact.hausdorff.spaces} is compatible with the concept of uniform continuity.

\begin{lem}[\cite{Isbell}]\label{lemma:continuity.implies.uniform.continuity} Let $X$ be a compact Hausdorff space and let $Y$ be a uniform space. If $f \colon X \to Y$ is continuous, then $f$ is uniformly continuous. \end{lem}

The last class of examples to be discussed is the class of topological groups. To this end, let $G$ be an arbitrary topological group. We denote by $\mathcal{V}(G)$ the filter of all neighborhoods of the neutral element in $G$. We define $G_{r}$ to be the uniform space obtained by endowing $G$ with the \emph{right uniformity}, i.e., $\{ \mathcal{U} \in \mathcal{P}(\mathcal{P}(G)) \mid \exists V \in \mathcal{V}(G) \colon \, \mathcal{U} \preceq \{ Vx \mid x \in G \} \}$. It is easy to see that the topology generated by the right uniformity is precisely the original topology of $G$. Besides, let us note that an injective homomorphism is given by the map $\lambda_{G} \colon G \to \Aut (G_{r})$ where $\lambda_{G}(g) (x) \defeq gx$ for all $g,x \in G$. For more details concerning uniform structures on topological groups, we refer to \cite{roelcke}.

\section{Amenability}\label{section:amenability}

In this section we want to recall the very basics concerning means on function spaces in general and invariant means on dynamical systems in particular. For this purpose, we follow the presentation in \cite{AnalysisOnSemigroups}.

Let $X$ be a set. The set of all finite subsets of $X$ shall be denoted by $\mathcal{F}(X)$. Additionally, let $\mathcal{F}_{+}(X) \defeq \mathcal{F}(X)\setminus \{ \emptyset \}$. Moreover, we denote by $B(X)$ the set of all bounded real-valued functions on $X$. Let $H$ be a linear subspace of $B(X)$. A \emph{mean} on $H$ is a linear map $m \colon H \to \mathbb{R}$ such that $\inf \{ f(x) \mid x \in X \} \leq m(f) \leq \sup \{ f(x) \mid x \in X \}$ for all $f \in H$.

Next we shall briefly introduce the concept of amenability for general dynamical systems. This will particularly capture the case of topological groups.

\begin{definition} Let $(X,G)$ be a \emph{dynamical system}, i.e., a pair consisting of a uniform space $X$ and a subgroup $G$ of $\Aut (X)$. An \emph{invariant mean} on $(X,G)$ is a mean $m \colon UC_{b}(X) \to \mathbb{R}$ such that $m(f) = m(f \circ g)$ for all $f \in UC_{b}(X)$ and $g \in G$. We say that $(X,G)$ is \emph{amenable} if there exists an invariant mean on $(X,G)$. \end{definition}

As pointed out at the end of Section~\ref{section:uniform.spaces}, any topological group may be considered as a uniform space, wherefore the previous definition particularly applies to topological groups.

\begin{definition}\label{definition:amenable.topological.group} A topological group $G$ is called \emph{amenable} if the dynamical system $(G_{r},\lambda_{G}(G))$ is amenable, i.e., there exists a mean $m \colon UC_{b}(G_{r}) \to \mathbb{R}$ such that $m(f) = m(f \circ \lambda_{G}(g))$ for all $f \in UC_{b}(G_{r})$ and $g \in G$. \end{definition}

For a more elaborate study of the concept of amenable for topological groups, the reader is referred to \cite{brownc,runde,paterson}.

\section{Matchings in bipartite graphs}\label{section:matchings}

In this section we shall provide some auxiliary results concerning matchings in bipartite graphs. Particular attention shall be paid to bipartite graphs arising from coverings. However, let us start with the general concepts and results.

\begin{definition}  By a \emph{bipartite graph}, we mean a triple $\mathcal{B} = (X,Y,R)$ consisting of two finite sets $X$ and $Y$ and a relation $R \subseteq X \times Y$. Let $\mathcal{B} = (X,Y,R)$ be a bipartite graph. If $S \subseteq X$, then we define $N_{\mathcal{B}}(S) \defeq \{ y \in Y \mid \exists x \in S \colon (x,y) \in R \}$. A \emph{matching} in $\mathcal{B}$ is an injective map $\phi \colon D \to Y$ such that $D \subseteq X$ and $(x,\phi(x)) \in R$ for all $x \in D$. A matching $\phi$ in $\mathcal{B}$ is said to be \emph{perfect} if $\dom (\phi) = X$. Furthermore, we call $\nu (\mathcal{B}) \defeq \sup \{ |\dom \phi | \mid \phi \textnormal{ matching in } \mathcal{B} \}$ the \emph{matching number} of $\mathcal{B}$. \end{definition}

\begin{thm}[\cite{Hall35}, \cite{Ore}]\label{theorem:hall} If $\mathcal{B} = (X,Y,R)$ is a bipartite graph, then \begin{displaymath}
	\nu (\mathcal{B}) = |X| - \sup \{ |S| - |N_{\mathcal{B}}(S)| \mid S \subseteq X \} .
\end{displaymath} \end{thm}

\begin{cor}\label{corollary:hall} If $\mathcal{B} = (X,Y,R)$ is a bipartite graph, then the following are equivalent: \begin{enumerate}
	\item[(1)]	$\mathcal{B}$ admits a perfect matching.
	\item[(2)]	$|S| \leq |N_{\mathcal{B}}(S)|$ for every subset $S \subseteq X$.
\end{enumerate} \end{cor}

In what follows, we shall have a closer look at bipartite graphs arising from coverings. However, let us first address some rather general bits of notation. If $X$ is a set, then we define $\id_{X} \colon X \to X, \, x \mapsto x$, and we denote by $\Sym (X)$ the group of all bijective maps from $X$ to $X$. For maps $f \colon X \to Y_{0}$ and $g \colon Y_{1} \to Z$, we define $g \circ f \colon f^{-1}(Y_{1}) \to Z, \, x \mapsto g(f(x))$. What is more, concerning a single map $f \colon X \to Y$, we define $f^{0} \defeq \id_{X} \colon X \to X$ and $f^{n+1} \defeq f \circ f^{n} \colon (f^{n})^{-1}(X) \to Y$ for every $n \in \mathbb{N}$. For later use, let us furthermore mention the following observation.

\begin{lem}\label{lemma:auxiliary.matchings} Let $X$ and $Y$ be sets and let $f \colon X \to Y$ be injective. For all $m,n \in \mathbb{N}\setminus \{ 0 \}$, $x \in \dom (f^{m})$ and $y \in \dom (f^{n})$, the following implication holds: \begin{displaymath}
	(\{ x,y\} \cap Y = \emptyset \wedge f^{m}(x) = f^{n}(y)) \ \Longrightarrow \ (m=n \wedge x=y) .
\end{displaymath} \end{lem}

\begin{proof} Let $m,n \in \mathbb{N}\setminus \{ 0 \}$, $x \in \dom (f^{m})$ and $y \in \dom (f^{n})$ such that $\{ x,y\} \cap Y = \emptyset$ and $f^{m}(x) = f^{n}(y)$. Suppose that $m < n$. Then $k \defeq n-m \geq 1$ and so $f^{k}(y) \in Y$. On the other hand, $f^{k}(y) = x \notin Y$, which clearly constitutes a contradiction. In a symmetric manner, the assumption that $n < m$ yields a contradiction as well. Hence, $m = n$ and thus $x = y$ as $f$ is injective. \end{proof}

Now we are going to deal with bipartite graphs induced by coverings.

\begin{definition} Let $X$ be a set, $\mathcal{U} \subseteq \mathcal{P}(X)$ and $E,F \in \mathcal{F}(X)$. We consider the bipartite graph $\mathcal{B} (E,F,\mathcal{U}) \defeq (E,F,R(E,F,\mathcal{U}))$ where \begin{displaymath}
	R(E,F,\mathcal{U}) \defeq \{ (x,y) \in E \times F \mid y \in \St (x,\mathcal{U}) \} = \{ (x,y) \in E \times F \mid \exists U \in \mathcal{U} \colon \, \{ x,y \} \subseteq U \}.
\end{displaymath} Furthermore, we define $\mu (E,F,\mathcal{U}) \defeq \nu (\mathcal{B}(E,F,\mathcal{U}))$. \end{definition}

The subsequent two lemmata will prove useful in Section~\ref{section:asymptotic.uniform.complexity}.

\begin{lem}\label{lemma:equivalent.matchings} Let $X$ be a set, let $\mathcal{U}$ be a covering of $X$, let $E \in \mathcal{F}(\Sym (X))$ and $F_{0},F_{1} \in \mathcal{F}(X)$. If $\Phi \colon F_{1} \to F_{0}$ is a bijection such that $x \in \bigcap \{ U \in \mathcal{U} \wedge \bigwedge\nolimits_{g \in E}g^{-1}(\mathcal{U}) \mid \Phi(x) \in U \}$ for all $x \in F_{1}$, then $\mu (F_{1},g(F_{1}),\mathcal{U}) \geq \mu (F_{0},g(F_{0}),\mathcal{U})$ for every $g \in E$. \end{lem}

\begin{proof} Let $g \in E$. Assume $\psi_{0}$ to be a matching in $\mathcal{B}(F_{0},g(F_{0}),\mathcal{U})$ where $|D_{0}| = \mu (F_{0},g(F_{0}),\mathcal{U})$ for $D_{0} \defeq \dom (\psi_{0})$. Let $D_{1} \defeq \Phi^{-1}(D_{0})$ and \begin{displaymath}
	\psi_{1} \colon D_{1} \to g(F_{1}), \, x \mapsto g(\Phi^{-1}(g^{-1}(\psi_{0}(\Phi (x))))) .
\end{displaymath} Evidently, $\psi_{1}$ is injective. We show that $\psi_{1}$ is a matching in $\mathcal{B}(F_{1},g(F_{1}),\mathcal{U})$. To this end, let $x \in D_{1}$. Since $\Phi(x) \in D_{0}$, there exists $U \in \mathcal{U}$ such that $\{ \Phi(x), \psi_{0}(\Phi(x)) \} \subseteq U$. That is, $\Phi(x) \in U$ and $g^{-1}(\psi_{0}(\Phi(x))) \in g^{-1}(U)$. By hypothesis on $\Phi$, it follows that $x \in U$ and $\Phi^{-1}(g^{-1}(\psi_{0}(\Phi(x)))) \in g^{-1}(U)$. Thus, $\{ x,\psi_{1}(x) \} = \{ x, g(\Phi^{-1}(g^{-1}(\psi_{0}(\Phi(x))))) \} \subseteq U$. This proves our claim. Consequently, $\mu (F_{1},g(F_{1}),\mathcal{U}) \geq |D_{1}| = |D_{0}| = \mu (F_{0},g(F_{0}),\mathcal{U})$. \end{proof}

\begin{lem}\label{lemma:compatible.matchings} Let $X$ be a set and let $\mathcal{U}$ be a covering of $X$. If $E,F \in \mathcal{F}(X)$ and $n \in \mathbb{N}\setminus \{ 0 \}$, then there exists a matching $\phi$ in $\mathcal{B}(E\setminus F,F\setminus E,\mathcal{U}^{\ast, n-1})$ such that \begin{displaymath}
	|\dom (\phi)| \geq \mu (E,F,\mathcal{U}) - \left(1+\frac{1}{n}\right)|E \cap F|.
\end{displaymath} \end{lem}

\begin{proof} Let $E,F \in \mathcal{F}(X)$ and $n \in \mathbb{N}\setminus \{ 0 \}$. Assume $\phi_{0}$ to be a matching in $\mathcal{B}(E,F,\mathcal{U})$ such that $|D_{0}| = \mu (E,F,\mathcal{U})$ where $D_{0} \defeq \dom (\phi_{0})$. We consider the function \begin{displaymath}
	\nu \colon D_{0} \to \{ 1,\ldots, n+1 \} ,\, x \mapsto \sup \{ m \in  \{ 1,\ldots, n+1 \} \mid x \in \dom (\phi_{0}^{m}) \} .
\end{displaymath} Define $D_{1} \defeq \{ x \in D_{0}\setminus F \mid \nu (x) \leq n, \, \phi_{0}^{\nu(x)}(x) \notin E \}$ and $\phi_{1} \colon D_{1} \to F\setminus E, \, x \mapsto \phi_{0}^{\nu(x)}(x)$. By Lemma~\ref{lemma:auxiliary.matchings}, $\phi_{1}$ is injective. We claim that $\phi_{1}$ is a matching in $\mathcal{B}(E\setminus F,F\setminus E,\mathcal{U}^{\ast, n-1})$. To prove this, let $x \in D_{1}$ and put $m \defeq \nu(x)$. By assumption, $\phi_{0}^{i}(x) \in \St (\phi_{0}^{i-1}(x),\mathcal{U})$ for each $i \in \{ 1,\ldots,m \}$, which implies that \begin{displaymath}
	\phi_{1}(x) = \phi_{0}^{m}(x) \in \St (\phi_{0}^{0}(x),\mathcal{U}^{\ast, m-1}) = \St (x,\mathcal{U}^{\ast, m-1}) \subseteq \St (x,\mathcal{U}^{\ast, n-1}) .
\end{displaymath} This proves our claim. We are left to show that $|\dom (\phi_{1})| \geq \mu (E,F,\mathcal{U}) - \left(1+\frac{1}{n}\right)|E \cap F|$. For this purpose, let \begin{align*}
	&S \defeq \{ x \in D_{0}\setminus F \mid \nu(x) \leq n, \, \phi_{0}^{\nu(x)}(x) \in E \} , &T \defeq \{ x \in D_{0}\setminus F \mid \nu(x) = n+1 \} .
\end{align*} Evidently, $D_{0}\setminus D_{1} = (D_{0} \cap F) \cup S \cup T$. By Lemma~\ref{lemma:auxiliary.matchings}, $\psi \colon S \to (E\cap F)\setminus D_{0}, \, x \mapsto \phi_{0}^{\nu(x)}(x)$ is injective, which implies that $|S| \leq |(E\cap F)\setminus D_{0}|$. According to Lemma~\ref{lemma:auxiliary.matchings}, it follows that $\pi \colon T \times \{ 1,\ldots,n \} \to D_{0} \cap F, \, (x,i) \to \phi_{0}^{i}(x)$ is injective. Hence, $|T| \leq \frac{1}{n}|D_{0} \cap F| \leq \frac{1}{n}|E \cap F|$. Therefore, \begin{align*}
	|D_{0}\setminus D_{1}| &\leq |D_{0} \cap F| + |(E\cap F)\setminus D_{0}| + \frac{1}{n}|E \cap F| \\
	&= |D_{0} \cap E \cap F| + |(E\cap F)\setminus D_{0}| + \frac{1}{n}|E \cap F| \\
	&= |E \cap F| + \frac{1}{n}|E \cap F| = \left(1+\frac{1}{n}\right)|E \cap F| .
\end{align*} Hence, $|D_{1}| \geq |D_{0}| - \left(1+\frac{1}{n}\right)|E \cap F|$, that is, $|\dom (\phi_{1})| \geq \mu (E,F,\mathcal{U}) - \left(1+\frac{1}{n}\right)|E \cap F|$. This completes the proof. \end{proof}

Finally in this section, we want to briefly point out the connection between matchings and uniform complexities of finite uniform coverings of uniform spaces.

\begin{lem}\label{lemma:uniform.refinement.map} Let $X$ be a uniform space. Furthermore, let $\mathcal{U},\mathcal{V},\mathcal{W} \in \mathcal{N}(X)$ such that $\mathcal{U} \preceq \mathcal{V}$, $\mathcal{U} \preceq \mathcal{W}$, and $N(\mathcal{U}) = |\mathcal{V}|$. There exists an injective map $\phi \colon \mathcal{V} \to \mathcal{W}$ such that $V \cap \phi (V) \ne \emptyset$ for all $V \in \mathcal{V}$. \end{lem}

\begin{proof} We shall utilize Corollary~\ref{corollary:hall}. For this purpose, we consider the bipartite graph $\mathcal{B} \defeq (\mathcal{V},\mathcal{W},R)$ where $R \defeq \{ (V,W) \in \mathcal{V} \times \mathcal{W} \mid V \cap W \ne \emptyset \}$. Let $\mathcal{S} \subseteq \mathcal{V}$ and $\mathcal{T} \defeq N_{\mathcal{B}}(\mathcal{S})$. We show that $|\mathcal{S}| \leq |\mathcal{T}|$. To this end, we argue that $\mathcal{Z} \defeq (\mathcal{V}\setminus \mathcal{S})\cup \mathcal{T}$ is an element of $\mathcal{N}(X)$. Evidently, $\mathcal{Z}$ is finite. In order to establish that $\mathcal{Z}$ is a uniform covering of $X$, we are going to prove that $\mathcal{V} \vee \mathcal{W}$ refines $\mathcal{Z}$. For this purpose, let $V \in \mathcal{V}$ and $W \in \mathcal{W}$ where $V \cap W \ne \emptyset$. Clearly, if $V \in \mathcal{V}\setminus \mathcal{S}$, then $V \cap W \subseteq V \in \mathcal{V}\setminus \mathcal{S} \subseteq \mathcal{Z}$. Otherwise, $W \in \mathcal{T}$ and hence $V \cap W \subseteq W \in \mathcal{T} \subseteq \mathcal{Z}$. Therefore, $(\mathcal{V} \vee \mathcal{W}) \setminus \{ \emptyset \}$ refines $\mathcal{Z}$. As $X$ is non-empty, it is true that $\mathcal{V} \vee \mathcal{W}$ refines $(\mathcal{V} \vee \mathcal{W})\setminus \{ \emptyset \}$, which implies that $\mathcal{V} \vee \mathcal{W}$ refines $\mathcal{Z}$. Thus, $\mathcal{Z} \in \mathcal{N}(X)$. Now, since $\mathcal{U} \preceq \mathcal{Z}$, it follows that $|\mathcal{V}| \leq |\mathcal{Z}|$ and hence $|\mathcal{S}| \leq |\mathcal{T}|$. Consequently, Corollary~\ref{corollary:hall} asserts that $\mathcal{B}$ admits a perfect matching. That is, there exists an injective map $\phi \colon \mathcal{V} \to \mathcal{W}$ such that $V \cap \phi(V) \ne \emptyset$ for all $V \in \mathcal{V}$. This proves our claim. \end{proof}

\section{Mean topological matching number}\label{section:mean.topological.matching.number}

In this section we want to summarize the main concepts and results of \cite{SchneiderThom}, where a novel invariant for general dynamical systems is introduced.

\begin{definition}[\cite{SchneiderThom}]\label{definition:mean.topological.matching.number} If $(X,G)$ is a dynamical system, then we define the \emph{mean topological matching number} of $(X,G)$ to be the quantity \begin{displaymath}
	\mu (X,G) \defeq \inf_{E \in \mathcal{F}(G)} \inf_{\mathcal{U} \in \mathcal{N}(X)} \sup_{F \in \mathcal{F}_{+}(X)} \inf_{g \in E} \frac{\mu (F,g(F),\mathcal{U})}{|F|} \in [0,1] .
\end{displaymath} \end{definition}

Recall that a topological space is \emph{perfect} if it does not contain any isolated points. Utilizing Theorem~\ref{theorem:hall} as well as the fact that any finite uniform covering of a uniform space admits a uniformly continuous partition of the unity (see \cite{Isbell}), the following result has been established.

\begin{thm}[\cite{SchneiderThom}]\label{theorem:mean.topological.matching.number} Let $(X,G)$ be a dynamical system. If $\mu (X,G) = 1$, then $(X,G)$ is amenable. Conversely, if $(X,G)$ is amenable and $X$ is a perfect Hausdorff space, then $\mu (X,G) = 1$. \end{thm}

This leads to a characterization of amenability for Hausdorff topological groups.

\begin{definition}[\cite{SchneiderThom}] If $G$ is a topological group, then we define the \emph{mean topological matching number} of $G$ to be the quantity \begin{displaymath}
	\mu (G) \defeq \mu (G_{r},\lambda_{G}(G)) = \inf_{E \in \mathcal{F}(G)} \inf_{\mathcal{U} \in \mathcal{N}(G_{r})} \sup_{F \in \mathcal{F}_{+}(G)} \inf_{g \in E} \frac{\mu (F,gF,\mathcal{U})}{|F|} .
\end{displaymath} \end{definition}

\begin{thm}[\cite{SchneiderThom}]\label{theorem:amenable.groups} Let $G$ be a topological group. If $\mu (G) = 1$, then $G$ is amenable. Conversely, if $G$ is amenable and Hausdorff, then $\mu (G) = 1$. \end{thm}

Since non-discrete topological groups are perfect, Theorem~\ref{theorem:amenable.groups} for non-discrete Hausdorff topological groups immediately follows from Theorem~\ref{theorem:mean.topological.matching.number}. However, in case of discrete groups, Theorem~\ref{theorem:amenable.groups} is a consequence of Theorem~\ref{theorem:mean.topological.matching.number} combined with F\o lner's celebrated characterization of discrete amenable groups \cite{folner}.

\section{Asymptotic uniform complexity}\label{section:asymptotic.uniform.complexity}

In this section we establish a novel invariant for general dynamical systems, which we shall call the \emph{asymptotic uniform complexity}. Subsequently, we explore the connection of this quantity to the mean topological matching number (Theorem~\ref{theorem:first.main.theorem} and Theorem~\ref{theorem:second.main.theorem}) and infer a corresponding characterization of amenability (Corollary~\ref{corollary:amenable}). The obtained results particularly apply to topological groups (Corollary~\ref{corollary:groups}).

\begin{definition}\label{definition:asymptotic.uniform.complexity} If $(X,G)$ is a dynamical system, then we define the \emph{asymptotic uniform complexity} of $(X,G)$ to be the quantity \begin{displaymath}
	\left. \omega (X,G) \defeq \sup_{E \in \mathcal{F}(G)} \sup_{\mathcal{U} \in \mathcal{N}(X)} \inf \left\{ \sup_{g \in E} \frac{N\left(\mathcal{V} \vee g(\mathcal{V})\right)}{N(\mathcal{V})} \, \right| \mathcal{V} \in \mathcal{N}(X), \, \mathcal{U} \preceq \mathcal{V} \right\} .
\end{displaymath} \end{definition}

Evidently, the asymptotic uniform complexity ranges in the interval $[1,\infty)$. However, it is not difficult to prove that it just takes values in $[1,2]$.

\begin{prop}\label{proposition:asymptotic.uniform.complexity} If $X$ is a uniform space, then $\omega (X,\Aut (X)) \leq 2$. \end{prop}

\begin{proof} Let $P(x) \defeq \bigcap \{ \St (x,\mathcal{U}) \mid \mathcal{U} \in \mathcal{N}(X) \}$ for each $x \in X$. According to Corollary~\ref{corollary:uniform.partition}, $\mathcal{P} \defeq \{ P(x) \mid x \in X \}$ constitutes a partition of $X$. Our proof of Proposition~\ref{proposition:asymptotic.uniform.complexity} proceeds by case analysis depending on whether $\mathcal{P}$ is finite.

First let us treat the case where $\mathcal{P}$ is finite. We argue that $\mathcal{P}$ is a member of $\mathcal{N}(X)$. To this end, let $F \in \mathcal{F}_{+}(X)$ such that $\mathcal{P} = \{ P(x) \mid x \in F \}$ and $|F| = |\mathcal{P}|$. Since $F$ is finite, there exists some $\mathcal{U} \in \mathcal{N}(X)$ such that $y \notin \St (x,\mathcal{U})$ for any two distinct elements $x,y \in F$. By Lemma~\ref{lemma:uniform.star.refinements}, there exists $\mathcal{V} \in \mathcal{N}(X)$ such that $\mathcal{U} \preceq^{\ast} \mathcal{V}$. We prove that $\mathcal{V}$ refines $\mathcal{P}$. For this purpose, let $V \in \mathcal{V}$. If $V = \emptyset$, then $\mathcal{P} \preceq \{ V \}$ as $\mathcal{P} \ne \emptyset$. Otherwise, there exists $x \in F$ such that $V \cap P(x) \ne \emptyset$, which implies that $V \cap \St(x,\mathcal{V}) \ne \emptyset$. We claim that $V \subseteq P(x)$. Let $y \in V$. There exists $z \in F$ such that $y \in P(z)$. Consequently, $y \in \St(z,\mathcal{V})$ and hence $V \cap \St(z,\mathcal{V}) \ne \emptyset$. Since $V \cap \St(x,\mathcal{V}) \ne \emptyset$, we conclude that $\{ x,z \} \subseteq \St (V,\mathcal{V})$. As $\mathcal{U} \preceq^{\ast} \mathcal{V}$, this clearly implies that $x = z$ and thus $y \in P(x)$. This proves that $V \subseteq P(x)$. Accordingly, $\mathcal{V}$ refines $\mathcal{P}$, whence $\mathcal{P}$ is a member of $\mathcal{N}(X)$. Furthermore, $g(\mathcal{P}) = \mathcal{P}$ and hence \begin{displaymath}
	\frac{N(\mathcal{P} \vee g(\mathcal{P}))}{N(\mathcal{P})} = \frac{N(\mathcal{P})}{N(\mathcal{P})} = \frac{|\mathcal{P}|}{|\mathcal{P}|} = 1 
\end{displaymath} for all $g \in \Aut(X)$. Since $\mathcal{U} \preceq \mathcal{P}$ for every $\mathcal{U} \in \mathcal{N}(X)$, this implies that $\omega (X,\Aut(X)) = 1$. This captures the case where $\mathcal{P}$ is finite.

Henceforth, assume $\mathcal{P}$ to be infinite. Let $\epsilon \in (0,\infty)$ and $\mathcal{U}_{0} \in \mathcal{N}(X)$. Since $\mathcal{P}$ is infinite, there exists $F \in \mathcal{F}_{+}(X)$ such that $|F| = |\{ P(x) \mid x \in F \}| \geq |\mathcal{U}_{0}|^{2}\epsilon^{-1}$. Utilizing Lemma~\ref{lemma:uniform.star.refinements} and the fact that $F$ is finite, we conclude that there exists some $\mathcal{U}_{1} \in \mathcal{N}(X)$ such that $\mathcal{U}_{0} \preceq^{\ast} \mathcal{U}_{1}$ and $y \notin \St (x,\mathcal{U}_{1})$ for any two distinct elements $x,y \in F$. According to Lemma~\ref{lemma:excision.of.uniform.coverings.a}, it follows that $\mathcal{V} \defeq \{ U\setminus F \mid U \in \mathcal{U}_{0} \} \cup \{ \St (x,\mathcal{U}_{1}) \mid x \in F \}$ is an element of $\mathcal{N}(X)$. Moreover, $\mathcal{V}$ refines $\mathcal{U}_{0}$. We show that $N(\mathcal{V}) \geq |F|$. Evidently, $F \cap (U \setminus F) = \emptyset$ for all $U \in \mathcal{U}_{0}$ and $F \cap \St (x,\mathcal{U}_{1}) = \{ x \}$ for all $x \in F$. Hence, $N(\mathcal{V}) \geq |F|$. Let $g \in \Aut(X)$. We argue that $N(\mathcal{V} \vee g(\mathcal{V})) \leq 2|F| + \epsilon|F|$. Let $\mathcal{W}_{0} \in \mathcal{N}(X)$ such that $\mathcal{V} \vee g(\mathcal{V}) \preceq^{\ast} \mathcal{W}_{0}$. By Lemma~\ref{lemma:excision.of.uniform.coverings.a}, \begin{displaymath}
	\mathcal{W} \defeq \left\{ U\setminus (F \cup g(F)) \left| U \in \mathcal{U}_{0} \vee g(\mathcal{U}_{0}) \right\} \cup \{ \St (x,\mathcal{W}_{0}) \mid x \in F \cup g(F) \} \right.
\end{displaymath} is a member of $\mathcal{N}(X)$. We are left to show that $\mathcal{V} \vee g(\mathcal{V}) \preceq \mathcal{W}$. To this end, we observe that $\left( U \cap g(V) \right)\setminus (F \cup g(F)) = (U\setminus F) \cap g(V\setminus F) \in \mathcal{V} \vee g(\mathcal{V})$ for all $U,V \in \mathcal{U}_{0}$, and thus $\{ U\setminus (F \cup g(F)) \mid U \in \mathcal{U}_{0} \vee g(\mathcal{U}_{0}) \} \subseteq \mathcal{V} \vee g(\mathcal{V})$. Moreover, $\mathcal{V} \vee g(\mathcal{V}) \preceq \{ \St (x,\mathcal{W}_{0}) \mid x \in F \cup g(F) \}$ because $\mathcal{V} \vee g(\mathcal{V}) \preceq^{\ast} \mathcal{W}_{0}$. It follows that $\mathcal{V} \vee g(\mathcal{V}) \preceq \mathcal{W}$. Consequently, \begin{displaymath}
	N\left(\mathcal{V} \vee g(\mathcal{V})\right) \leq |\mathcal{W}| \leq 2|F| + |\mathcal{U}_{0}|^{2} \leq 2|F| + \epsilon |F| ,
\end{displaymath} and thus \begin{displaymath}
	\frac{N\left(\mathcal{V} \vee g(\mathcal{V})\right)}{N(\mathcal{V})} \leq \frac{2|F| + \epsilon |F|}{|F|} = 2 + \epsilon .
\end{displaymath} This substantiates that $\omega (X,G) \leq 2$. \end{proof}

Subsequently, we are going to point out the close connection between the quantity introduced in Definition~\ref{definition:asymptotic.uniform.complexity} and the mean topological matching number (see Definition~\ref{definition:mean.topological.matching.number}). Our first objective is to establish a lower bound for the former in terms of the latter.

\begin{thm}\label{theorem:first.main.theorem} If $(X,G)$ is a dynamical system, then $\omega (X,G) \geq 2 - \mu (X,G)$. \end{thm}

\begin{proof} Let $\epsilon \in (0,\infty)$, $E \in \mathcal{F}(G)$ and $\mathcal{U} \in \mathcal{N}(X)$. By Lemma~\ref{lemma:uniform.star.refinements}, there exists $\mathcal{U}_{0} \in \mathcal{N}(X)$ such that $\mathcal{U} \preceq^{\ast} \mathcal{U}_{0}$. Let $\mathcal{U}_{1} \in \mathcal{N}(X)$ such that $\sup_{g \in E} N(\mathcal{U}_{1} \vee g(\mathcal{U}_{1})) \leq (\omega(X,G) + \epsilon)N(\mathcal{U}_{1})$ and $\mathcal{U}_{0} \vee \bigvee_{g \in E} g^{-1}(\mathcal{U}_{0}) \preceq \mathcal{U}_{1}$. Let $\mathcal{V} \in \mathcal{N}(X)$ where $\mathcal{U}_{1} \preceq \mathcal{V}$ and $|\mathcal{V}| = N(\mathcal{U}_{1})$. Clearly, $N(\mathcal{V}) = |\mathcal{V}|$. According to Lemma~\ref{lemma:point.projection.map}, there is an injective map $\pi \colon \mathcal{V} \to X$ such that $\pi (V) \in V$ for all $V \in \mathcal{V}$. Define $F \defeq \pi (\mathcal{V})$. We are going to show that \begin{displaymath}
		\inf_{g \in E} \mu (F,g(F),\mathcal{U}) \geq (2-\omega(X,G) - \epsilon ) |F| .
\end{displaymath} To this end, let $g \in E$. There is $\mathcal{W} \in \mathcal{N}(X)$ such that $\mathcal{U}_{1} \vee g(\mathcal{U}_{1}) \preceq \mathcal{W}$ and $|\mathcal{W}| = N(\mathcal{U}_{1} \vee g(\mathcal{U}_{1}))$. Since $\mathcal{W}$ refines $\mathcal{U}_{1}$, Lemma~\ref{lemma:uniform.refinement.map} asserts that there exists an injective map $\phi \colon \mathcal{V} \to \mathcal{W}$ such that $V \cap \phi (V) \ne \emptyset$ for all $V \in \mathcal{V}$. Moreover, since $g^{-1}(\mathcal{W})$ refines $\mathcal{U}_{1}$, Lemma~\ref{lemma:uniform.refinement.map} implies that there exists an injective map $\psi \colon \mathcal{V} \to g^{-1}(\mathcal{W})$ such that $V \cap \psi (V) \ne \emptyset$ for all $V \in \mathcal{V}$. Define $D \defeq \pi (\phi^{-1}(g(\psi (\mathcal{V}))))$ and $\rho \colon D \to g(F), \, x \mapsto g(\pi(\psi^{-1}(g^{-1}(\phi(\pi^{-1}(x))))))$. Evidently, $\rho$ is injective. We argue that $\rho$ constitutes a matching in $\mathcal{B}(F,g(F),\mathcal{U})$. To this end, let $x \in D$. Now, $x \in \pi^{-1}(x)$ by hypothesis on $\pi$. Besides, $\pi^{-1}(x) \cap \phi(\pi^{-1}(x)) \ne \emptyset$ by hypothesis on $\phi$. As $x \in D$, we conclude that $g^{-1}(\phi(\pi^{-1}(x))) \in \psi(\mathcal{V})$. By assumption about $\psi$, it follows that $g^{-1}(\phi(\pi^{-1}(x))) \cap \psi^{-1}(g^{-1}(\phi (\pi^{-1}(x)))) \ne \emptyset$, thus $\phi(\pi^{-1}(x)) \cap g(\psi^{-1}(g^{-1}(\phi (\pi^{-1}(x))))) \ne \emptyset$. Additionally, $\rho (x) \in g(\psi^{-1}(g^{-1}(\phi (\pi^{-1}(x)))))$. Since $\mathcal{U} \preceq^{\ast} \mathcal{U}_{0}$ and $\mathcal{U}_{0} \vee \bigvee_{g \in E} g^{-1}(\mathcal{U}_{0}) \preceq \mathcal{U}_{1}$, we may conclude that there exists some $U \in \mathcal{U}$ such that $\{ x,\rho(x) \} \subseteq U$. This proves our claim. Finally, we establish that \begin{align*}
	|D| &= |\phi (\mathcal{V}) \cap g(\psi (\mathcal{V}))| \geq |\mathcal{W}| - |\mathcal{W}\setminus \phi (\mathcal{V})| - |\mathcal{W}\setminus g(\psi (\mathcal{V}))| = |\mathcal{W}| - 2(|\mathcal{W}| - |\mathcal{V}|) \\
	&= 2|\mathcal{V}| - |\mathcal{W}| \geq 2|\mathcal{V}| - (\omega (X,G) + \epsilon )|\mathcal{V}| = (2-\omega(X,G) - \epsilon ) |F| ,
\end{align*} wherefore $\mu (F,g(F),\mathcal{U}) \geq |D| \geq (2-\omega(X,G) - \epsilon ) |F|$. This proves our claim. We conclude that $\mu (X,G) \geq 2 - \omega(X,G) - \epsilon$ for all $\epsilon \in (0,\infty)$. This readily implies that $\mu (X,G) \geq 2 - \omega(X,G)$, which completes the proof. \end{proof}

We aim at providing an exact equation relating the asymptotic uniform complexity of a perfect Hausdorff dynamical system to its mean topological matching number (Theorem~\ref{theorem:second.main.theorem}). For this purpose, we shall utilize the following three auxiliary lemmata.

\begin{lem}\label{lemma:large.matchings} Let $X$ be a perfect Hausdorff space. Let $\mathcal{U}$ be a finite open covering of $X$, let $E$ be a finite set of homeomorphisms of $X$ to itself, let $F_{0} \in \mathcal{F}(X)$ and $n \in \mathbb{N}$. Then there exists $F \in \mathcal{F}(X)$ such that $|F| = n|F_{0}|$ and $\mu (F,g(F),\mathcal{U}) \geq n \mu (F_{0},g(F_{0}),\mathcal{U})$ for each $g \in E$. \end{lem}

\begin{proof} For each $x \in F_{0}$, the set $U(x) \defeq \bigcap \{ U \in \mathcal{U} \wedge \bigwedge_{g \in E} g^{-1}(\mathcal{U}) \mid x \in U \}$ constitutes an open neighborhood of $x$ in $X$. Since $X$ is a perfect Hausdorff space, every open non-empty subset of $X$ is infinite. Hence, there is $\Phi \colon F_{0} \to \mathcal{F}(X)$ such that \begin{enumerate}
	\item[(1)]	$\forall x \in F_{0} \colon \, \Phi(x) \subseteq U(x)$,
	\item[(2)]	$\forall x \in F_{0} \colon \, |\Phi(x)| = n$,
	\item[(3)]	$\forall x,y \in F_{0} \colon \, x \ne y \Longrightarrow \Phi(x) \cap \Phi(y) = \emptyset$.
\end{enumerate} Let $F \defeq \bigcup \{ \Phi(x) \mid x \in F_{0} \}$. We observe that $|F| = n|F_{0}|$. Let $g \in E$. Assume $\psi_{0}$ to be a matching in $\mathcal{B}(F_{0},g(F_{0}),\mathcal{U})$ such that $|D_{0}| = \mu (F_{0},g(F_{0}),\mathcal{U})$ where $D_{0} \defeq \dom (\psi_{0})$. Let us define $D \defeq \bigcup \{ \Phi(x) \mid x \in D_{0} \}$. Note that $|D| = n|D_{0}|$ by (2) and (3). Furthermore, due to (2) and (3), there exists an injective map $\psi \colon D \to g(F)$ such that $\psi (\Phi (x)) = g(\Phi(g^{-1}(\psi_{0}(x))))$ for all $x \in D_{0}$. We argue that $\psi$ constitutes a matching in $\mathcal{B} (F,g(F),\mathcal{U})$. To this end, let $z \in D$. There exists $x \in D_{0}$ where $z \in \Phi (x)$. Furthermore, there exists $U \in \mathcal{U}$ such that $\{ x, \psi_{0}(x) \} \subseteq U$, which means that $x \in U$ and $g^{-1}(\psi_{0}(x)) \in g^{-1}(U)$. Due to (1), the first assertion readily implies that $z \in U$. Besides, $\psi (z) \in g(\Phi(g^{-1}(\psi_{0}(x))))$. Therefore, $g^{-1}(\psi(z)) \in \Phi(g^{-1}(\psi_{0}(x)))$ and hence $g^{-1}(\psi (z)) \in g^{-1}(U)$ by (1). Thus, $\{ z,\psi(z) \} \subseteq U$. This proves our claim. Accordingly, $\mu (F,g(F),\mathcal{U}) \geq |D| = n|D_{0}| = n \mu (F_{0},g(F_{0}),\mathcal{U})$. \end{proof}

\begin{lem}\label{lemma:matchings.vs.fixed.points} Let $X$ be a perfect Hausdorff topological space, let $Y$ be a dense subset of $X$, let $\mathcal{U}$ be a finite open covering of $X$, let $E \in \mathcal{F}(\Sym(X))$ and $F_{0} \in \mathcal{F}(X)$. Then there exists $F_{1} \in \mathcal{F}(Y)$ and a bijection $\Phi \colon F_{1} \to F_{0}$ such that \begin{enumerate}
	\item[(1)]	$\forall x \in F_{1} \colon \, x \in \bigcap \{ U \in \mathcal{U} \mid \Phi(x) \in U \}$,
	\item[(2)]	$\forall g \in E \, \forall x \in F_{1} \cap g(F_{1}) \colon \, g(x) = x$.
\end{enumerate} \end{lem}

\begin{proof} Denote by $\mathcal{Z}$ the set of all pairs $(F_{1},\Phi)$ consisting of a finite subset $F_{1} \subseteq Y$ and an injective map $\Phi \colon F_{1} \to F_{0}$ and satisfying the conditions (1) and (2). Let $(F_{1},\Phi) \in \mathcal{Z}$ such that $|F_{1}| = \sup \{ |F_{2}| \mid (F_{2},\Psi) \in \mathcal{Z} \}$. We claim that $|F_{1}| = |F_{0}|$. On the contrary, assume that $|F_{1}| < |F_{0}|$. Then there exists $y \in F_{0}\setminus \Phi(F_{1})$. Since $X$ is a perfect Hausdorff space, every open non-empty subset of $X$ is infinite. Consequently, there exists some \begin{displaymath}
	x \in Y \cap \left( \left(\bigcap \{ U \in \mathcal{U} \mid y \in U \}\right) \setminus \left( F_{1} \cup \bigcup\nolimits_{g \in E} g(F_{1}) \cup g^{-1}(F_{1}) \right) \right) .
\end{displaymath} We define $F_{2} \defeq F_{1} \cup \{ x \}$ and $\Psi \colon F_{2} \to F_{0}$ such that $\Psi|_{F_{1}} = \Phi$ and $\Psi(x) = y$. We observe that $(F_{2},\Psi)$ is a member of $\mathcal{Z}$. Since $|F_{1}| < |F_{2}|$, this clearly contradicts our hypothesis. Hence, $|F_{1}| = |F_{0}|$ and therefore $\Phi$ is a bijection. This finishes the proof. \end{proof}

\begin{lem}\label{lemma:large.matchings.vs.fixed.points} Let $X$ be a perfect Hausdorff uniform space, let $Y$ be a dense subset of $X$, and let $G$ be a subgroup of $\Aut(X)$. Furthermore, let $\mathcal{U} \in \mathcal{N}(X)$, $E \in \mathcal{F}(G)$, $\epsilon \in (0,\infty)$ and $n \in \mathbb{N}$. Then there exists $F \in \mathcal{F}_{+}(Y)$ such that \begin{enumerate}
	\item[(1)]	$\inf_{g \in E} \mu (F,g(F),\mathcal{U}) \geq (\mu (X,G) - \epsilon)|F|$,
	\item[(2)]	$|F| \geq n$,
	\item[(3)]	$\forall g \in E \, \forall x \in F \cap g(F) \colon \, g(x) = x$.
\end{enumerate} \end{lem}

\begin{proof} By Lemma~\ref{lemma:uniform.star.refinements}, there exists an open covering $\mathcal{V} \in \mathcal{N}(X)$ such that $\mathcal{U} \preceq \mathcal{V}$. Let $F_{0} \in \mathcal{F}_{+}(G)$ such that $\inf_{g \in E} \mu (F_{0},g(F_{0}),\mathcal{V}) \geq (\mu (X,G) - \epsilon)|F_{0}|$. By Lemma~\ref{lemma:large.matchings}, there exists $F_{1} \in \mathcal{F}_{+}(X)$ such that $|F_{1}| = n|F_{0}|$ and $\mu (F_{1},g(F_{1}),\mathcal{V}) \geq n \mu (F_{0},g(F_{0}),\mathcal{V})$ for each $g \in E$. Due to Lemma~\ref{lemma:matchings.vs.fixed.points} and Lemma~\ref{lemma:equivalent.matchings}, there exists some $F \in \mathcal{F}_{+}(Y)$ such that $|F| = |F_{1}|$ as well as $\mu (F,g(F),\mathcal{V}) \geq \mu (F_{1},g(F_{1}),\mathcal{V})$ and $F \cap g(F) \subseteq \{ x \in X \mid g(x) = x \}$ for every $g \in G$. Evidently, the conditions (2) and (3) are satisfied. Furthermore, we conclude that \begin{align*}
	\inf_{g \in E} \mu (F,g(F),\mathcal{U}) &\geq \inf_{g \in E} \mu (F,g(F),\mathcal{V}) \geq \inf_{g \in E} \mu (F_{1},g(F_{1}),\mathcal{V}) \geq n \inf_{g \in E} \mu (F_{0},g(F_{0}),\mathcal{V}) \\
	&\geq n(\mu (X,G) - \epsilon)|F_{0}| = (\mu (X,G) - \epsilon)|F_{1}| = (\mu (X,G) - \epsilon)|F| .
\end{align*} This proves the claim. \end{proof}

Now, everything is prepared to deduce the aforementioned result for perfect Hausdorff dynamical systems.

\begin{thm}\label{theorem:second.main.theorem} Let $X$ be a perfect Hausdorff uniform space. If $G$ is a subgroup of $\Aut(X)$, then $\omega (X,G) = 2-\mu (X,G)$. \end{thm}

\begin{proof} According to Theorem~\ref{theorem:first.main.theorem}, we just need to prove that $\omega (X,G) \leq 2-\mu (X,G)$. To this end, let $\epsilon \in (0,\infty)$, $E \in \mathcal{F}(G)$ and $\mathcal{U}_{0} \in \mathcal{N}(X)$. Put $\theta \defeq \mu (X,G)$ and $n \defeq \left\lceil \frac{4}{\epsilon} \right\rceil$. By Lemma~\ref{lemma:uniform.star.refinements}, there exists an open covering $\mathcal{U}_{1} \in \mathcal{N}(X)$ such that $\mathcal{U}_{0} \preceq^{\ast} \mathcal{U}_{1}$. Let \begin{displaymath}
	\mathcal{U}_{2} \defeq \mathcal{U}_{1} \vee \bigvee\nolimits_{g \in E} g(\mathcal{U}_{1}) .
\end{displaymath} By Lemma~\ref{lemma:uniform.star.refinements}, there exists $\mathcal{U}_{3} \in \mathcal{N}(X)$ such that $\mathcal{U}_{2} \preceq \mathcal{U}_{3}^{\ast, n-1}$. By Lemma~\ref{lemma:large.matchings.vs.fixed.points}, there exists $F \in \mathcal{F}_{+}(X)$ such that \begin{enumerate}
	\item[(1)] 		$\inf_{g \in E} \mu (F,g(F),\mathcal{U}_{3}) \geq \left(\theta-\frac{\epsilon}{4}\right)|F|$,
	\item[(2)] 		$|F| \geq \frac{2|\mathcal{U}_{1}|^{2}}{\epsilon}$,
	\item[(3)]		$\forall g \in E \, \forall x \in F \cap g(F) \colon \, g(x) = x$.
\end{enumerate} If $x \in F$, then $U(x) \defeq \St (x,\mathcal{U}_{1}) \setminus (F\setminus \{ x \} )$ is an open neighborhood of $x$ in $X$. Let $\mathcal{V} \defeq \{ U\setminus F \mid U \in \mathcal{U}_{1} \} \cup \{ U(x) \mid x \in F \}$. According to Lemma~\ref{lemma:excision.of.uniform.coverings.b}, $\mathcal{V}$ is a member of $\mathcal{N}(X)$. Moreover, $\mathcal{U}_{0} \preceq \{ U\setminus F \mid U \in \mathcal{U}_{1} \}$ and $\mathcal{U}_{0} \preceq \mathcal{U}_{1}^{\ast} \preceq \{ U(x) \mid x \in F \}$. Therefore, $\mathcal{U}_{0} \preceq \mathcal{V}$. We show that $N(\mathcal{V}) \geq |F|$. Evidently, $F \cap (U \setminus F) = \emptyset$ for all $U \in \mathcal{U}_{1}$ and $F \cap U(x) = \{ x \}$ for all $x \in F$. Hence, $N(\mathcal{V}) \geq |F|$. Let $g \in E$. We argue that $N(\mathcal{V} \vee g(\mathcal{V})) \leq (2-\theta )|F| + \epsilon|F|$. Due to Lemma~\ref{lemma:compatible.matchings}, there exist $D \subseteq F$ and an injective map $\phi \colon D \to g(F)$ such that \begin{enumerate}
	\item[(4)] 		$|D| \geq \mu (F,g(F),\mathcal{U}_{3}) - \frac{1}{n}|F|$,
	\item[(5)]		$\forall x \in F \cap g(F) \colon \, x \in D \wedge \phi(x) = x$,
	\item[(6)]		$\forall x \in D \colon \, \phi(x) \in \St (x,\mathcal{U}_{2})$.
\end{enumerate} Let $x \in D$. Of course, $V(x) \defeq U(x) \cap g(U(g^{-1}(\phi(x))))$ is open. We claim that $x \in V(x)$. By (6), there exists $U \in \mathcal{U}_{2}$ such that $\{ x,\phi(x) \} \subseteq U$, which implies that $\mathcal{U}_{1} \preceq \{ g^{-1}(U) \}$ and $\{ g^{-1}(x),g^{-1}(\phi(x)) \} \subseteq g^{-1}(U)$. Thus, $g^{-1}(x) \in \St(g^{-1}(\phi(x)),\mathcal{U}_{1})$. Besides, if $g^{-1}(x) \in F$, then $g(x) = x$ by (3) and $\phi(x) = x$ by (5). Hence, $g^{-1}(x) \in U(g^{-1}(\phi(x)))$. This shows that $x \in V(x)$. Let us consider the set \begin{displaymath}
	S \defeq F \cup g(F) = D \cup (F\setminus D) \cup \phi (D) \cup (g(F)\setminus \phi (D)) .
\end{displaymath} Of course, there exist $\mathcal{W}_{0},\mathcal{W}_{1} \subseteq \mathcal{V} \vee g(\mathcal{V})$ such that $F\setminus D \subseteq \bigcup \mathcal{W}_{0}$, $g(F)\setminus \phi (D) \subseteq \bigcup \mathcal{W}_{1}$, $|\mathcal{W}_{0}| \leq |F\setminus D|$, and $|\mathcal{W}_{1}| \leq |g(F)\setminus \phi (D)|$. Let \begin{displaymath}
	\mathcal{W} \defeq \{ V(x) \mid x \in D \} \cup \mathcal{W}_{0} \cup \mathcal{W}_{1} \cup \left\{ U\setminus S \left| U \in \mathcal{U}_{1} \vee g(\mathcal{U}_{1}) \right\} .\right.
\end{displaymath} We prove that $S \subseteq \bigcup_{x \in D} V(x) \cup \bigcup (\mathcal{W}_{0} \cup \mathcal{W}_{1})$. Clearly, $(F\setminus D) \cup (g(F)\setminus \phi (D)) \subseteq \bigcup (\mathcal{W}_{0} \cup \mathcal{W}_{1})$ and $D \subseteq \bigcup_{x \in D} V(x)$. So, let $y \in \phi(D)$. Then $x \defeq \phi^{-1}(y) \in D$. We are going to show that $y \in V(x)$. Obviously, $y = g(g^{-1}(y)) \in g(U(g^{-1}(y))) = g(U(g^{-1}(\phi(x))))$. Besides, by (6), there is $U \in \mathcal{U}_{2}$ such that $\{ x,y \} \subseteq U$. Hence, as $\mathcal{U}_{1} \preceq \mathcal{U}_{2}$, we conclude that $y \in \St (x,\mathcal{U}_{1})$. Moreover, if $y \in F$, then $g(y) = y$ by (3) and $\phi(y) = y$ by (5), which implies that $x = y$. Consequently, $y \in U(x)$ and thus $y \in V(x)$. This proves that $S \subseteq \bigcup_{x \in D} V(x) \cup \bigcup (\mathcal{W}_{0} \cup \mathcal{W}_{1})$. By Lemma~\ref{lemma:excision.of.uniform.coverings.b}, $\mathcal{W}$ is an element of $\mathcal{N}(X)$. It remains to be shown that $\mathcal{V} \vee g(\mathcal{V}) \preceq \mathcal{W}$. Clearly, $\mathcal{V} \vee g(\mathcal{V}) \preceq \{ V(x) \mid x \in D \} \cup \mathcal{W}_{0} \cup \mathcal{W}_{1}$. Moreover, $\left( U \cap g(V) \right)\setminus S = (U\setminus F) \cap g(V\setminus F) \in \mathcal{V} \vee g(\mathcal{V})$ for all $U,V \in \mathcal{U}_{1}$, and thus $\{ U\setminus S \mid U \in \mathcal{U}_{1} \vee g(\mathcal{U}_{1}) \} \subseteq \mathcal{V} \vee g(\mathcal{V})$. It follows that $\mathcal{V} \vee g(\mathcal{V}) \preceq \mathcal{W}$. Consequently, \begin{align*}
	N\left(\mathcal{V} \vee g(\mathcal{V})\right) &\leq |\mathcal{W}| \leq |D| + |F\setminus D| + |g(F)\setminus \phi (D)| + |\mathcal{U}_{1}|^{2} \\
	&\stackrel{(2)}{\leq} |F| + |F\setminus D| + \frac{\epsilon }{2}|F| \stackrel{(4)}{\leq} |F| + |F| - \mu (F,g(F),\mathcal{U}_{3}) + \frac{1}{n}|F| + \frac{\epsilon }{2}|F| \\
	&\stackrel{(1)}{\leq} 2|F| - \left(\theta-\frac{\epsilon}{4}\right)|F| + \frac{1}{n}|F| + \frac{\epsilon }{2}|F| \leq (2-\theta )|F| + \epsilon|F| ,
\end{align*} and thus \begin{displaymath}
	\frac{N\left(\mathcal{V} \vee g(\mathcal{V})\right)}{N(\mathcal{V})} \leq \frac{(2-\theta )|F| + \epsilon|F|}{|F|} = 2-\theta + \epsilon .
\end{displaymath} This substantiates that $\omega (X,G) \leq 2 -\theta$. \end{proof}

\begin{cor}\label{corollary:amenable} Let $(X,G)$ be a dynamical system. If $\omega (X,G) = 1$, then $(X,G)$ is amenable. Conversely, if $X$ is a perfect Hausdorff space and $(X,G)$ amenable, then $\omega (X,G) = 1$. \end{cor}

\begin{proof} This follows from Theorem~\ref{theorem:first.main.theorem}, Theorem~\ref{theorem:mean.topological.matching.number}, and Theorem~\ref{theorem:second.main.theorem}. \end{proof}

Finally in this section, we shall have a closer look at the asymptotic uniform complexity of the dynamical system associated to an arbitrary topological group in Definition~\ref{definition:amenable.topological.group}. This will lead to a novel characterization of amenability for non-discrete Hausdorff topological groups.

\begin{definition} If $G$ is a topological group, then we define the \emph{asymptotic uniform complexity} of $G$ to be the quantity \begin{displaymath}
	\left. \omega (G) \defeq \omega (G_{r},\lambda_{G}(G)) = \sup_{E \in \mathcal{F}(G)} \sup_{\mathcal{U} \in \mathcal{N}(G_{r})} \inf \left\{ \sup_{g \in E} \frac{N\left(\mathcal{V} \vee g(\mathcal{V})\right)}{N(\mathcal{V})} \, \right| \mathcal{V} \in \mathcal{N}(G_{r}), \, \mathcal{U} \preceq \mathcal{V} \right\} .
\end{displaymath} \end{definition}

\begin{cor}\label{corollary:groups} Suppose $G$ to be an arbitrary non-discrete Hausdorff topological group. Then $\omega (G) = 2-\mu(G)$. Moreover, $G$ is amenable if and only if $\omega(G) = 1$. \end{cor}

\begin{proof} Note that $G$ is perfect because $G$ is homogeneous and not discrete. Accordingly, Theorem~\ref{theorem:second.main.theorem} asserts that $\omega (G) = 2-\mu(G)$. Moreover, due to Corollary~\ref{corollary:amenable}, $G$ is amenable if and only if $\omega(G) = 1$. \end{proof}

\section{Topologically free dynamical systems}\label{section:topologically.free.dynamical.systems}

In this section we shall establish an alternative characterization of amenability applying to topologically free dynamical systems (Theorem~\ref{theorem:third.main.theorem}), which particularly applies to topological groups (Corollary~\ref{corollary:third.main.theorem}).

\begin{definition} Let $X$ be a uniform space. A subgroup $G$ of $\Aut (X)$ (or the dynamical system $(X,G)$, resp.) is called \emph{topologically free} if $\{ x \in X \mid \forall g \in E\setminus \{ \id_{X} \} \colon \, g(x) \ne x \}$ is dense in $X$ for every $E \in \mathcal{F}(G)$. \end{definition}

A short moment of reflection reveals the following observation.

\begin{remark}\label{remark:topological.freeness} If $G$ is a topological group, then $\lambda_{G}(G)$ is a topologically free subgroup of $\Aut (G_{r})$ (cf.~Definition~\ref{definition:amenable.topological.group}). \end{remark}

Moreover, topologically free dynamical systems play an important role in the theory of semigroup compactifications (see \cite{AnalysisOnSemigroups}). To explain this, let $S$ be a compact Hausdorff left-topological monoid. Consider the group $G \defeq \{ x \in S \mid \exists y \in S \colon \, xy=yx=1_{S} \}$ as well as the injective homomorphism $\lambda \colon G \to \Aut (S)$ where $\lambda(g)(x) \defeq gx$ for all $g \in G$ and $s \in S$. Now, if $G$ is dense in $S$, then $\lambda (G)$ is a topologically free subgroup of $\Aut(S)$. Hence, semigroup compactifications of topological groups in the sense of \cite{AnalysisOnSemigroups} give rise to topologically free dynamical systems.

\begin{thm}\label{theorem:third.main.theorem} Let $X$ be a perfect Hausdorff uniform space and let $G$ be a topologically free subgroup of $\Aut (X)$. Then $(X,G)$ is amenable if and only if \begin{displaymath}
	\left. \sup_{E \in \mathcal{F}(G)} \sup_{\mathcal{U} \in \mathcal{N}(X)} \inf \left\{ \frac{N\left(\bigvee_{g \in E} g(\mathcal{V})\right)}{N(\mathcal{V})} \, \right| \mathcal{V} \in \mathcal{N}(X), \, \mathcal{U} \preceq \mathcal{V} \right\} = 1 .
\end{displaymath} \end{thm}

\begin{proof} ($\Longleftarrow $) By assumption, it follows that $\omega(X,G) = 1$. Consequently, Corollary~\ref{corollary:amenable} asserts that $(X,G)$ is amenable.

($\Longrightarrow $) Suppose $(X,G)$ to be amenable. By Theorem~\ref{theorem:mean.topological.matching.number}, it follows that $\mu(X,G) = 1$. Now, let $\epsilon \in (0,\infty)$, $E \in \mathcal{F}(G)$ and $\mathcal{U}_{0} \in \mathcal{N}(X)$. Put \begin{displaymath}
	\delta \defeq \frac{\epsilon}{4|E|^{2}+2}
\end{displaymath} and $n \defeq \left\lceil \delta^{-1} \right\rceil$. By Lemma~\ref{lemma:uniform.star.refinements}, there exists an open covering $\mathcal{U}_{1} \in \mathcal{N}(X)$ such that $\mathcal{U}_{0} \preceq \mathcal{U}_{1}^{\ast \ast}$. Let \begin{displaymath}
	\mathcal{U}_{2} \defeq \mathcal{U}_{1} \vee \bigvee\nolimits_{g \in E} g(\mathcal{U}_{1}) .
\end{displaymath} By Lemma~\ref{lemma:uniform.star.refinements}, there exists $\mathcal{U}_{3} \in \mathcal{N}(X)$ such that $\mathcal{U}_{2} \preceq \mathcal{U}_{3}^{\ast, n-1}$. Since $(X,G)$ is topologically free, the set $Y \defeq \{ x \in X \mid \forall g \in (E^{-1}E)\setminus \{ \id_{X} \} \colon \, g(x) \ne x \}$ is dense in $X$. By Lemma~\ref{lemma:large.matchings.vs.fixed.points}, there exists $F \in \mathcal{F}_{+}(Y)$ such that \begin{enumerate}
	\item[(1)] 		$\inf_{g \in E} \mu (F,g(F),\mathcal{U}_{3}) \geq (1-\delta)|F|$,
	\item[(2)] 		$|F| \geq \frac{2|\mathcal{U}_{1}|^{|E|+1}}{\epsilon}$,
	\item[(3)]		$\forall g \in E^{-1}E \, \forall x \in F \cap g(F) \colon \, g(x) = x$.
\end{enumerate} If $x \in F$, then $U(x) \defeq \St(x,\mathcal{U}_{1}^{\ast}) \setminus (F\setminus \{ x \} )$ is an open neighborhood of $x$ in $X$. Let $\mathcal{V} \defeq \{ U\setminus F \mid U \in \mathcal{U}_{1} \} \cup \{ U(x) \mid x \in F \}$. According to Lemma~\ref{lemma:excision.of.uniform.coverings.b}, $\mathcal{V}$ is a member of $\mathcal{N}(X)$. Moreover, $\mathcal{U}_{0} \preceq \{ U\setminus F \mid U \in \mathcal{U}_{1} \}$ and $\mathcal{U}_{0} \preceq \mathcal{U}_{1}^{\ast \ast} \preceq \{ U(x) \mid x \in F \}$. Therefore, $\mathcal{U}_{0} \preceq \mathcal{V}$. We show that $N(\mathcal{V}) \geq |F|$. Evidently, $F \cap (U \setminus F) = \emptyset$ for all $U \in \mathcal{U}_{1}$ and $F \cap U(x) = \{ x \}$ for all $x \in F$. Hence, $N(\mathcal{V}) \geq |F|$. We argue that \begin{displaymath}
	N\left(\bigvee\nolimits_{g \in E} g(\mathcal{V})\right) \leq |F| + \epsilon|F| .
\end{displaymath} Applying Lemma~\ref{lemma:compatible.matchings}, we obtain the following: for each $g \in E$, there exist $D_{g} \subseteq F$ and an injective map $\phi_{g} \colon D_{g} \to g(F)$ such that \begin{enumerate}
	\item[(4)] 		$|D_{g}| \geq \mu (F,g(F),\mathcal{U}_{3}) - \frac{1}{n}|F|$,
	\item[(5)]		$\forall x \in F \cap g(F) \colon \, x \in D_{g} \wedge \phi_{g}(x) = x$,
	\item[(6)]		$\forall x \in D \colon \, \phi_{g}(x) \in \St (x,\mathcal{U}_{2})$.
\end{enumerate} Let $D \defeq \bigcap_{g \in E} D_{g}$. Note that $V(x) \defeq \bigcap\nolimits_{g \in E} g(U(g^{-1}(\phi_{g}(x))))$ is open in $X$ for each $x \in D$. Furthermore, consider the sets $S \defeq \bigcup\nolimits_{g \in E} g(F)$ and $T \defeq \bigcup\nolimits_{g \in E} g(F)\setminus \phi_{g}(D)$. Now, \begin{align*}
	|T| &\leq \sum\nolimits_{g \in E} |g(F)\setminus \phi_{g}(D)| = \sum\nolimits_{g \in E} |F\setminus D| \leq \sum\nolimits_{g \in E} \sum\nolimits_{h \in E} |F\setminus D_{h}| \\
	&\stackrel{(4)}{\leq} |E|^{2}((1+\delta )|F| - \mu (F,h(F),\mathcal{U}_{3})) \stackrel{(1)}{\leq} |E|^{2}((1+\delta ) - (1-\delta ))|F| = 2|E|^{2}\delta |F| \leq \frac{\epsilon }{2}|F|.
\end{align*} Of course, there exists $\mathcal{W}_{0} \subseteq \bigvee_{g \in E} g(\mathcal{V})$ such that $T \subseteq \bigcup \mathcal{W}_{0}$ and $|\mathcal{W}_{0}| \leq |T|$. Let \begin{displaymath}
	\mathcal{W} \defeq \{ V(x) \mid x \in D \} \cup \mathcal{W}_{0} \cup \left\{ U\setminus S \, \left| U \in \mathcal{U}_{2} \right\} \right.
\end{displaymath} We prove that $S \subseteq \bigcup_{x \in D} V(x) \cup \bigcup \mathcal{W}_{0}$. Evidently, $T \subseteq \bigcup \mathcal{W}_{0}$. Henceforth, let $y \in S\setminus T$. There exists $g \in E$ such that $y \in \phi_{g}(D)$, that is, $x \defeq \phi_{g}^{-1}(y) \in D$. We show that $y \in V(x)$. To this end, let $h \in E$. By (6), there exists $U \in \mathcal{U}_{2}$ where $\{ x,y \} \subseteq U$, which implies that $\mathcal{U}_{1} \preceq \{ h^{-1}(U) \}$ and $\{ h^{-1}(x),h^{-1}(y) \} \subseteq h^{-1}(U)$. Moreover, once more due to (6), there exists some $V \in \mathcal{U}_{2}$ such that $\{ x,\phi_{h}(x) \} \subseteq V$, which means that $\mathcal{U}_{1} \preceq \{ h^{-1}(V) \}$ and $\{ h^{-1}(x),h^{-1}(\phi_{h}(x)) \} \subseteq h^{-1}(V)$. As it is true that $h^{-1}(U) \cap h^{-1}(V) \ne \emptyset$, we conclude that $\mathcal{U}_{1}^{\ast} \preceq \{ h^{-1}(U) \cup h^{-1}(V) \}$. Since $\{ h^{-1}(y), h^{-1}(\phi_{h}(x)) \} \subseteq h^{-1}(U) \cup h^{-1}(V)$, it follows that $h^{-1}(y) \in \St(h^{-1}(\phi_{h}(x)),\mathcal{U}_{1}^{\ast})$. Besides, if $h^{-1}(y) \in F$, then $h^{-1}(g(g^{-1}(y))) = g^{-1}(y)$ by (3) and hence $g = h$ as $g^{-1}(y) \in Y$, which implies that $y = \phi_{h}(x)$. Thus, $y \in h(U(h^{-1}(\phi_{h}(x))))$. Consequently, $y \in V(x)$. This proves that $S \subseteq \bigcup_{x \in D} V(x) \cup \bigcup \mathcal{W}_{0}$. According to Lemma~\ref{lemma:excision.of.uniform.coverings.b}, it follows that $\mathcal{W}$ is an element of $\mathcal{N}(X)$. We are left to show that $\bigvee_{g \in E} g(\mathcal{V}) \preceq \mathcal{W}$. Of course, $\bigvee_{g \in E} g(\mathcal{V}) \preceq \{ V(x) \mid x \in D \} \cup \mathcal{W}_{0}$. Moreover, \begin{displaymath}
	\left( \bigcap\nolimits_{g \in E} U_{g}\right)\setminus S = \bigcap\nolimits_{g \in E} g(U_{g}\setminus F) \in \bigvee\nolimits_{g \in E} g(\mathcal{V})
\end{displaymath} for all $(U_{g})_{g \in E} \in \mathcal{U}_{1}^{E}$. Thus, $\bigvee\nolimits_{g \in E} g(\mathcal{V}) \preceq \{ U\setminus S \mid U \in \mathcal{U}_{2} \}$ and hence $\bigvee\nolimits_{g \in E} g(\mathcal{V}) \preceq \mathcal{W}$. Consequently, \begin{align*}
	N\left(\bigvee\nolimits_{g \in E} g(\mathcal{V})\right) &\leq |\mathcal{W}| \leq |D| + |T| + |\mathcal{U}_{1}|^{|E|+1} \stackrel{(2)}{\leq} |F| + \frac{\epsilon}{2}|F| + \frac{\epsilon}{2}|F| = |F| + \epsilon |F| ,
\end{align*} and thus \begin{displaymath}
	\frac{N\left(\bigvee\nolimits_{g \in E} g(\mathcal{V})\right)}{N(\mathcal{V})} \leq \frac{|F| + \epsilon|F|}{|F|} = 1 + \epsilon .
\end{displaymath} This completes the proof. \end{proof}

Finally, let us explicitly explore the previous result for topological groups.

\begin{cor}\label{corollary:third.main.theorem} Let $G$ be a non-discrete Hausdorff topological group. Then $G$ is amenable if and only if \begin{displaymath}
	\left. \sup_{E \in \mathcal{F}(G)} \sup_{\mathcal{U} \in \mathcal{N}(G_{r})} \inf \left\{ \frac{N\left(\bigvee_{g \in E} g(\mathcal{V})\right)}{N(\mathcal{V})} \, \right| \mathcal{V} \in \mathcal{N}(G_{r}), \, \mathcal{U} \preceq \mathcal{V} \right\} = 1 .
\end{displaymath} \end{cor}

\begin{proof} As observed earlier, $G$ is perfect because $G$ is homogeneous and not discrete. Besides, Remark~\ref{remark:topological.freeness} asserts that $\lambda_{G}(G)$ is a topologically free subgroup of $\Aut (G_{r})$. Hence, the stated equivalence is an immediate consequence of Theorem~\ref{theorem:third.main.theorem}. \end{proof}

\section{Entropic dimension}\label{section:entropic.dimension}

In this section we provide a sufficient criterion for amenability concerning the group of Lipschitz-automorphisms of an arbitrary precompact pseudo-metric space in terms of its entropic dimension (Theorem~\ref{theorem:metric.case}). To this end, let us first recall some additional terminology concerning pseudo-metric spaces.

Let $X$ be a pseudo-metric space. Note that $X$ is precompact (cf.~Section~\ref{section:uniform.spaces}) if and only if, for each $r \in (0,\infty)$, there exists some $F \in \mathcal{F}(X)$ such that $X = \bigcup \{ B(x,r) \mid x \in F \}$. If $X$ is precompact, then we put \begin{displaymath}
	\gamma (r) \defeq \inf \left\{ |F| \left| \, F \in \mathcal{F}(X), \, X = \bigcup \{ B(x,r) \mid x \in F \} \right\} \right.
\end{displaymath} whenever $r \in (0,\infty)$, and we define the \emph{entropic dimension} of $X$ (as in \cite{cogrowth}) to be \begin{displaymath}
	\delta (X) \defeq \limsup_{r \to 0} - \frac{\log_{2}(\gamma (r))}{\log_{2}(r)} .
\end{displaymath} Let $Y$ be another pseudo-metric space. A map $f \colon X \to Y$ is said to be \emph{Lipschitz-continuous} if there exists $c \in [0,\infty)$ such that $d(f(x),f(y)) \leq cd(x,y)$ for all $x,y \in X$. A bijection $f \colon X \to Y$ is called a \emph{Lipschitz-isomorphism} if both $f$ and $f^{-1}$ are Lipschitz-continuous maps. By a \emph{Lipschitz-automorphism} of $X$, we mean a Lipschitz-isomorphism from $X$ to itself. The group of all Lipschitz-automorphisms of $X$ shall be denoted by $\Aut_{L}(X)$. Note that any Lipschitz-continuous map between pseudo-metric spaces is uniformly continuous with regard to the respective uniformities.

\begin{lem}\label{lemma:convergence} Let $(a_{n})_{n \in \mathbb{N}\setminus \{ 0 \}}$ be an increasing sequence in $[1,\infty)$. If $\limsup_{n \to \infty} \frac{\log_{2}(a_{n})}{n} = 0$, then $\liminf_{n \to \infty} \frac{a_{n+k}}{a_{n}} = 1$ for every $k \in \mathbb{N}$. \end{lem}

\begin{proof} Let $k \in \mathbb{N}$. The proof proceeds by contraposition. Assume that $\liminf_{n \to \infty} \frac{a_{n+k}}{a_{n}} > 1$. Then there exist $a \in (1,\infty)$ and $n_{0} \in \mathbb{N} \setminus \{ 0 \}$ such that $\frac{a_{n+k}}{a_{n}} \geq a$ for all $n \in \mathbb{N} \setminus \{ 0 \}$ with $n \geq n_{0}$. We conclude that $a_{mk+n_{0}} \geq a^{m}a_{n_{0}}$ and hence \begin{displaymath}
	\frac{\log_{2}(a_{mk+n_{0}})}{mk+n_{0}} \geq \frac{m\log_{2}(a) + \log_{2}(a_{n_{0}})}{mk+n_{0}} \geq \frac{m}{mk+n_{0}}\log_{2}(a)
\end{displaymath} for all $m \in \mathbb{N}$. Since $a > 1$, it follows that $\limsup_{n \to \infty} \frac{\log_{2}(a_{n})}{n} > 0$, which clearly constitutes a contradiction. \end{proof}

\begin{thm}\label{theorem:metric.case} Let $X$ be a pseudo-metric space. If $X$ is precompact and $\delta(X) = 0$, then $\omega (X,\Aut_{L}(X)) = 1$ and hence $(X,\Aut_{L}(X))$ is amenable. \end{thm}

\begin{proof} We show that $\omega (X,\Aut_{L}(X)) = 1$. To this end, let $\epsilon \in (0,\infty)$, $E \in \mathcal{F}(\Aut_{L}(X))$ and $\mathcal{U} \in \mathcal{N}(X)$. There exists $l \in \mathbb{N}$ such that $d(g^{-1}(x),g^{-1}(y)) \leq 2^{l} d(x,y)$ for all $x,y \in X$ and $g \in E$. Moreover, there exists some $m \in \mathbb{N}$ where $\mathcal{U} \preceq \{ B(x,2^{-m}) \mid x \in X \}$. By assumption, \begin{displaymath}
	\limsup_{n \to \infty} \frac{\log_{2}(\gamma (2^{-n}))}{n} = \limsup_{n \to \infty} -\frac{\log_{2}(\gamma (2^{-n}))}{\log_{2}(2^{-n})} \leq \delta (X) = 0 .
\end{displaymath} Due to Lemma~\ref{lemma:convergence}, there exists $n \in \mathbb{N}$ where $m \leq n$ and $\gamma(2^{-(n+l+2)}) \leq (1 + \epsilon)\gamma(2^{-n})$. Let $F \in \mathcal{F}(X)$ such that $|F| = \gamma (2^{-(n+l+2)})$ and $X = \bigcup \{ B(x,2^{-(n+l+2)}) \mid x \in F \}$. We argue that $\mathcal{V} \defeq \{ B(x,2^{-n}) \mid x \in F \}$ is a member of $\mathcal{N}(X)$. For every $x \in X$, there exists some $y \in F$ such that $x \in B(y,2^{-(n+l+2)}) \subseteq B(y,2^{-(n+1)})$, which readily implies that $B(x,2^{-(n+1)}) \subseteq B(y,2^{-n}) \in \mathcal{V}$. That is, $\mathcal{V} \preceq \{ B(x,2^{-(n+1)}) \mid x \in X \}$. Therefore, $\mathcal{V} \in \mathcal{N}(X)$. Evidently, $\mathcal{U} \preceq \mathcal{V}$ and $N(\mathcal{V}) \geq \gamma(2^{-n})$. Now, let $g \in E$. We are going to show that $N(\mathcal{V} \vee g(\mathcal{V})) \leq \gamma (2^{-(n+l+2)})$. For this purpose, let $\mathcal{W} \defeq \{ B(x,2^{-(n+l+1)}) \mid x \in F \}$. We claim that $\mathcal{W}$ is a member of $\mathcal{N}(X)$. Indeed, for each $x \in X$, there exists some $y \in F$ such that $x \in B(y,2^{-(n+l+2)})$, which implies that $B(x,2^{-(n+l+2)}) \subseteq B(y,2^{-(n+l+1)}) \in \mathcal{W}$. That is, $\mathcal{W} \preceq \{ B(x,2^{-(n+l+2)}) \mid x \in X \}$. Therefore, $\mathcal{W} \in \mathcal{N}(X)$. Moreover, for each $x \in F$, there exists some $y \in F$ such that $g^{-1}(x) \in B(y,2^{-(n+l+1)}) \subseteq B(y,2^{-(n+1)})$, which readily implies that $g^{-1}(B(x,2^{-(n+l+1)})) \subseteq B(g^{-1}(x),2^{-(n+1)}) \subseteq B(y,2^{-n})$ and therefore $B(x,2^{-(n+l+1)}) \subseteq B(x,2^{-n}) \cap g(B(y,2^{-n}))$. We conclude that $\mathcal{V} \vee g(\mathcal{V}) \preceq \mathcal{W}$. Consequently, $N(\mathcal{V} \vee g(\mathcal{V})) \leq |\mathcal{W}| \leq |F| = \gamma (2^{-(n+l+2)})$. Thus, \begin{displaymath}
	\frac{N(\mathcal{V} \vee g(\mathcal{V}))}{N(\mathcal{V})} \leq \frac{\gamma (2^{-(n+l+2)})}{\gamma (2^{-n})} \leq 1+\epsilon .
\end{displaymath} This proves that $\omega (X,\Aut_{L}(X)) = 1$. Hence, Corollary~\ref{corollary:amenable} asserts that $(X,\Aut_{L}(X))$ is amenable. \end{proof}

\section{Topological entropy}\label{section:topological.entropy}

In this rather short concluding section we provide a sufficient condition for amenability applying to those dynamical systems where the acting group is finitely generated. More precisely, we show that vanishing topological entropy implies amenability (Theorem~\ref{theorem:vanishing.entropy.implies.amenability}).

\begin{definition}\label{definition:topological.entropy} Let $(X,G)$ be a dynamical system. If $E$ is a finite symmetric generating subset of $G$ containing the identity map, then we define the \emph{topological entropy} of $(X,G)$ with respect to $E$ to be the quantity \begin{displaymath}
	\eta_{E}(X,G) \defeq \sup_{\mathcal{U} \in \mathcal{N}(X)} \limsup_{n \to \infty} \frac{\log_{2} N \left( \bigvee_{g \in E^{n}} g(\mathcal{U})\right)}{n} .
\end{displaymath} \end{definition}

Of course, the precise value of the quantity introduced in Definition~\ref{definition:topological.entropy} depends upon the choice of a generating subset. However, the following observation reveals that the hypothesis of Theorem~\ref{theorem:vanishing.entropy.implies.amenability} is not affected by this matter of choice.

\begin{remark}\label{remark:topological.entropy} Let $(X,G)$ be a dynamical system. Suppose $E_{0}$ and $E_{1}$ to be finite symmetric generating subsets of $G$ containing the identity map. Then \begin{displaymath}
	\frac{1}{m} \eta_{E_{1}}(X,G) \leq \eta_{E_{0}}(X,G) \leq n \eta_{E_{1}}(X,G) ,
\end{displaymath} where $m \defeq \inf \{ k \in \mathbb{N}\setminus \{ 0 \} \mid E_{1} \subseteq E_{0}^{k} \}$ and $n \defeq \inf \{ k \in \mathbb{N}\setminus \{ 0 \} \mid E_{0} \subseteq E_{1}^{k} \}$. \end{remark}

\begin{thm}\label{theorem:vanishing.entropy.implies.amenability} Let $(X,G)$ be a dynamical system and let $E$ be a finite symmetric generating subset of $G$ containing the identity map. If $\eta_{E}(X,G) = 0$, then $\omega (X,G) = 1$ and $(X,G)$ is amenable. \end{thm}

\begin{proof} Let $\epsilon \in (0,\infty)$, $E_{0} \in \mathcal{F}(G)$ and $\mathcal{U} \in \mathcal{N}(X)$. There exists some $m \in \mathbb{N}\setminus \{ 0 \}$ such that $E_{0} \subseteq E_{1} \defeq E^{m}$. By Remark~\ref{remark:topological.entropy}, $\eta_{E_{1}}(X,G) = 0$. Due to Lemma~\ref{lemma:convergence}, there exists $n \in \mathbb{N}\setminus \{ 0 \}$ such that \begin{displaymath}
	N\left(\bigvee\nolimits_{g \in E_{1}^{n+1}} g(\mathcal{U})\right) \leq (1 + \epsilon )N\left(\bigvee\nolimits_{g \in E_{1}^{n}} g(\mathcal{U})\right) .
\end{displaymath} Of course, $\mathcal{V} \defeq \bigvee_{g \in E_{1}^{n}} g(\mathcal{U})$ is a member of $\mathcal{N}(X)$ refining $\mathcal{U}$. We argue that \begin{displaymath}
	\sup_{g \in E} N\left(\mathcal{V} \vee g(\mathcal{V})\right) \leq (1 + \epsilon )N(\mathcal{V}) .
\end{displaymath} To this end, let $\mathcal{W} \in \mathcal{N}(X)$ where $\bigvee_{g \in E_{1}^{n+1}} g(\mathcal{U}) \preceq \mathcal{W}$ and $|\mathcal{W}| = N\left(\bigvee\nolimits_{g \in E_{1}^{n+1}} g(\mathcal{U})\right)$. Evidently, if $g \in E$, then $\mathcal{W}$ refines $\mathcal{V} \vee g(\mathcal{V})$, whence \begin{displaymath}
	N\left(\mathcal{V} \vee g(\mathcal{V})\right) \leq |\mathcal{W}| = N\left(\bigvee\nolimits_{g \in E_{1}^{n+1}} g(\mathcal{U})\right) \leq (1 + \epsilon) N(\mathcal{V}) .
\end{displaymath} We conclude that $\omega(X,G) = 1$. Due to Corollary~\ref{corollary:amenable}, it follows that $(X,G)$ is amenable. \end{proof}

%%%%%%%%%%%%%%%%%%%%%%%%%%%%%%%%%%

%\small
%\bibliographystyle{amsalpha}
%\bibliography{bibtexfile}

\printbibliography

\end{document}